\documentclass{article}

\usepackage{amssymb,amsmath, amsthm}

\newtheorem{theo}{Theorem}
\newtheorem{coro}{Corollary}[section]
\newtheorem{lemm}[coro]{Lemma}
\newtheorem{prop}[coro]{Proposition}
\newtheorem{rema}[coro]{Remark}
\newtheorem{defi}[coro]{Definition}
\newtheorem{ques}{Question}
\newtheorem{clai}{Claim}

   \def\DD{{\mathbb D}}

 \def\RR{{\mathbb R}} \def\SS{{\mathbb S}} \def\TT{{\mathbb T}}
   
 \def\ZZ{{\mathbb Z}}

\def\La{\Lambda}
\def\De{\Delta}

\def\Ga{\Gamma}

\def\cA{{\cal A}}    
  \def\cH{{\cal H}}  
\def\cC{{\cal C}}   \def\cO{{\cal O}} \def\cU{{\cal U}}
\def\cD{{\cal D}}   \def\cP{{\cal P}} 
    
\def\cF{{\cal F}}   \def\cR{{\cal R}}

\title{Volume hyperbolicity and wildness}

\author{Christian Bonatti and Katsutoshi Shinohara}

\begin{document}

\maketitle

\begin{abstract}
% \cite{DPU,BDP} 
It is known that \emph{volume hyperbolicity} (partial hyperbolicity and
 uniform expansion or contraction of the volume in the extremal bundles) is a necessary condition for 
 robust transitivity or robust chain recurrence hence for tameness. 
In this paper,  on any $3$-manifold we build examples of quasi-attractors
which are volume hyperbolic and wild at the same time.
As a main corollary, we see that, for any closed $3$-manifold $M$, 
the space  $\mathrm{Diff}^1(M)$ admits a non-empty
open set where every $C^1$-generic diffeomorphism 
has no attractors or repellers.
 
The main tool of our construction is the notion of
flexible periodic points introduced in \cite{BS}. 
For ejecting the flexible points from the quasi-attractor, we control
the topology of the quasi-attractor using the notion of  
filtrating Markov partition, which we introduce in this paper\footnote{This paper has been supported by the ANR project DynNonHyp, 
the project Ci\^encia Sem Fronteiras CAPES (Brazil) and
the JSPS KAKENHI Grant-in-Aid for JSPS Fellows (26$\cdot$1121).
We thank the warm hospitality of IMB (Institut de Math\'ematiques de Bourgogne), Dijon, France 
and Departamento de Matem\'{a}tica of PUC-Rio, 
Rio de Janeiro, Brazil.}.
\end{abstract}

{\footnotesize{2010 \emph{Mathematics Subject Classification}:  37C05; 37C15, 37C70, 37D45.}

\emph{Key words and phrases}: Wild dynamics, partial hyperbolicity,
volume hyperbolicity, quasi-attractor.}

%!TEX root = yetwild.tex

\section{Introduction}
\subsection{Backgrounds}

Hyperbolic dynamical systems are those which are nowadays
considered as well understood  (see \cite{BDV} for further information).  
However, they are far to represent 
the majority of dynamical systems: 
it is known from the late sixties that there are open sets in $\mathrm{Diff}^r(M)$ 
of non-uniformly hyperbolic systems. 
These systems are each unstable. Hence
it seems improbable that we obtain 
a complete topological classification of them. 

Nevertheless, some recent works propose a panorama of 
non-hyperbolic dynamical systems (see \cite{B,CP}). 
They attempt to characterize the robust non-hyperbolicity 
by local phenomena 
(for instance robust heterodimensional cycle or/and 
robust homoclinic tangencies) and 
weak hyperbolic structures (such as partial 
hyperbolicity or dominated splittings). 

As a first step of this global program,
there is an important division of the space of dynamical systems
into two classes of completely different nature:

\begin{itemize}
\item We say that a diffeomorphism is \emph{tame} 
if each chain recurrence class is $C^1$-robustly isolated.  
This implies that there are finitely many chain recurrence classes and that the number of classes 
does not vary under perturbations of the system.
 \item We say that a diffeomorphism is \emph{wild} if it admits a neighborhood in 
 $\mathrm{Diff}^1(M)$ 
in which $C^1$-generic systems have infinitely many chain recurrence classes.
\end{itemize}
The result in \cite{BC} shows that the set of tame diffeomorphisms and the set of wild diffeomorphisms are 
disjoint $C^1$-open sets whose union is dense in $\mathrm{Diff}^1(M)$ (see \cite{A} for a preliminary version, 
which only holds for $C^1$-generic diffeomorphism).  
 
While the uniform hyperbolicity is a sufficient condition for tameness,
there are examples of non-hyperbolic tame diffeomorphisms \cite{S,M,BD1,BV}. 
We also have examples of wild systems \cite{BD2,BD3,BLY}, which are necessarily non-hyperbolic. 
However, this division is far to be understood:
\begin{itemize}
 \item Every chain recurrence class (which is neither a sink nor a source) 
 of tame systems has a structure called {\it volume hyperbolicity}: 
 it consists of a finest dominated splitting $E_1\oplus_< E_2\oplus_<\cdots \oplus_<E_k$ whose extremal bundles $E_1$ and $E_k$ are volume contracting and expanding, respectively. 
See \cite{BDP, B}.
 \item All the known examples of wild diffeomorphisms use lack of dominated splittings which allow us to produce sinks or sources 
 (possibly after restriction of the dynamics 
 on a normally hyperbolic invariant submanifold). 
\end{itemize}

The aim of this paper is to approach the frontier between 
tame and wild dynamics,
by producing examples of wild diffeomorphisms 
under a setting of volume hyperbolicity, 
using a completely different method from the existing ones.
  
  \smallskip
  
We consider diffeomorphisms on $3$-manifolds which have a  
chain recurrence class $C(p)$ 
(where $p$ is a hyperbolic saddle of $s$-index $2$, i.e., $\dim(W^s(p))=2$)
admitting a partially hyperbolic splitting $E^{cs}\oplus_< E^u$, 
where $E^u$ is uniformly expanding and $E^{cs}$ is uniformly area contracting
(for the precise definitions of chain recurrence classes and partially hyperbolic 
splittings, see Section~\ref{ss.partialhyp}).  
We assume that $C(p)$ is not contained in a 
normally hyperbolic invariant surface 
(otherwise $p$ is an attracting 
point in restriction to this surface, and 
$C(p)$ is trivial). 
Note that, according to \cite{BC2}, this is 
equivalent to that there is
a periodic point $q$ in $C(p)$ for which the strong unstable manifold of 
$q$ meets $C(p)$ at a point $q'\neq q$.
This is the setting for the examples of robustly transitive diffeomorphisms in \cite{BV}
and of the robustly transitive attractor of \cite{Ca}. 
In this setting, we will nevertheless build examples of wild diffeomorphisms.

Following the definition of wild homoclinic classes in $\cite{B}$, we say that 
a chain recurrence class $C(p)$ of a diffeomorphism $f$ is ($C^1$-){\it wild} 
if it is locally generically non-isolated, 
that is, there exists a neighborhood $\mathcal{U} \subset \mathrm{Diff}^1(M)$ of $f$ 
and a residual set $\mathcal{R} \subset \mathcal{U}$
such that for every $g \in \mathcal{R}$ the continuation $C(p_g)$ is not isolated 
(every neighborhood of $C(p_g)$ has non-empty intersection with 
a chain recurrence class which is not equal to $C(p_g)$).
Note that having a wild chain recurrence class is  a sufficient condition for the wildness of a diffeomorphism. 

We will construct a diffeomorphism having a wild chain recurrence class under the above setting
with additional conditions.
Note that the volume hyperbolicity implies that there are 
neither sinks nor sources in 
a sufficiently small neighborhood of $C(p)$. 
For some of our examples, the class $C(p)$ will be the unique 
quasi-attractor in an attracting neighborhood, 
and its basin will cover a residual subset of this attracting neighborhood.

The main technique of our construction is the 
notion of {\it $\varepsilon$-flexible points} defined in \cite{BS}.
It is an $s$-index $2$ 
periodic point $x$ admitting an $\varepsilon$-perturbation 
(with respect to a fixed $C^1$-distance)
which changes $x$ into an $s$-index $1$ periodic orbit whose stable manifold 
being an arbitrarily chosen curve in the center-stable plane.  
In \cite{BS}, we also showed that the existence of flexible 
points with arbitrarily small $\varepsilon$
is a prevalent phenomenon under the presence of a robustly 
non-hyperbolic $2$-dimensional center-stable bundle 
with no dominated splitting. 
Therefore, these flexible points seem to be the good candidate 
for being ejected from the original chain recurrence class $C(p)$.
We only need to choose a curve out of the class $C(p)$ and give a
perturbation which turns the stable manifold of the periodic point
into the chosen one. 
Once we have the disjointness of 
a fundamental domain of the stable manifold 
from $C(p)$, then we can see that the orbit of 
$x$ is indeed out of $C(p)$. 

This simple argument transforms the question of knowing 
if $C(p)$ is tame or wild into an almost purely topological problem.
The question is: \emph{
in the local center-stable manifold passing through $x$,
does there exist a path
starting from $x$ which is disjoint from $C(p)$?} 
In the examples of tame dynamics in \cite{BV, Ca}, the class is either the whole manifold
or a Sierpinski carpet in the center-stable direction. These examples have flexible points. 
Meanwhile, 
the flexibility is not enough to ejecting the point, 
because there is no path going out of the class from the point. 

Here we give a setting in which the intersection of the 
chain recurrence class with the center stable 
manifold passing through the flexible point will be contained in a finite 
union of disjoint discs. This condition allows us to choose the 
stable manifold of this flexible point so that it avoids these discs. 
In such a situation, the flexibility of
the periodic point enables us to eject it
from the original class. 

Now, we state our result in more formal way. 

\subsection{Statement of the result}

We consider a diffeomorphism $f$ of a $3$-manifold.
We will define precisely in the next section the notions of 
\emph{partially hyperbolic filtrating Markov partition} 
of \emph{saddle type} (Definition~\ref{d.saddle}) or of \emph{attracting type} 
(Definition~\ref{d.attracting}).  Let us give a rough description of it. 

It consists of the following:
\begin{itemize}
 \item 
a Markov partition $\{R_i\}_{i=1,\dots,k}$ whose \emph{rectangles} 
$R_i$ are cylinders (we say that a subset of a 3-manifold is a \emph{cylinder}
if it is $C^1$-diffeomorphic to $\DD^2\times [0,1]$,
where $\DD^2$ is the unit disc in $\mathbb{R}^2$);
\item the union of the rectangles 
${\bf R}=\bigcup_{ i=1}^{k} R_i$ is a \emph{filtrating set}, that is, 
an intersection of an attracting region $A$  
and a repelling region $R$ 
(in the attracting type, the repelling region $R$ is the whole manifold, 
in other words, the union of the 
rectangles is an attracting region); 
\item the \emph{vertical boundary} of each rectangle 
(i.e.,  the side of the cylinder $(\partial \DD^2)\times [0,1]=\SS^1\times [0,1]$) 
is contained in the boundary of the attracting region $\partial A$;
\item for the saddle type filtrating Markov partition 
we require that the \emph{horizontal boundaries} 
$\DD^2\times \partial [0,1]$
are contained in the boundary of the repelling region $\partial R$;
\item there is a partially hyperbolic structure $E^{cs}\oplus_< E^u$ defined
on the union of the rectangles: 
$Df$ (the differential map of $f$) leaves invariant a continuous \emph{unstable cone field} $\cC^u$
 % (vectors are uniformly expanded)
 defined on the union of the rectangles,
 transverse to the horizontal discs $\DD^2\times \{t\}$ and containing the
 vertical lines $\{p\}\times [0,1]$. 
\end{itemize}

We say that a hyperbolic periodic 
point $x\in  R_i$  with  \emph{$s$-index} $2$ 
has \emph{large
stable manifold} if there is a local stable manifold of $x$ which is a disc contained in $R_i$ 
(necessarily transverse 
to the unstable cone field) whose boundary is contained in the vertical 
boundary of $R_i$.  
Finally, recall that two hyperbolic periodic points 
$p$ and $q$ are said to be
{\it homoclinically related} if $W^u(p)$ and $W^s(q)$ have 
non-empty transverse intersection and the same holds for
$W^u(q)$ and $W^s(p)$.

We are now ready to state our main result.
\begin{theo}\label{t.wild} 
Let $f$ be a diffeomorphism of a $3$-manifold, 
having a partially hyperbolic filtrating 
Markov partition ${\bf R}=\bigcup_{ i=1}^{k} R_i$ of either saddle type or attracting type.  
Assume that there is a periodic point $p\in{\bf R}$ 
and the following holds:
\begin{itemize}
 \item $p$ is an $s$-index $2$ hyperbolic periodic point with 
large stable manifold;
\item there is an $s$-index $2$ hyperbolic periodic point $p_1$ homoclinically related with $p$
having a complex (non-real) stable eigenvalue;
\item there is a periodic point $q$ of $s$-index $1$ which 
is $C^1$-robustly in the chain recurrence class $C(p,f)$ of $p$: 
for every $g$ which is sufficiently $C^1$-close to $f$ one has 
$$C(p_g,g)=C(q_g,g),$$ 
where $p_g$ and $q_g$ are the continuation of $p$ and $q$ for $g$.  
\end{itemize}
Then, there are a $C^1$-diffeomorphism $g$ arbitrarily 
$C^1$-close to $f$, 
hyperbolic periodic points $x_g$ and $y_g$ (of $g$)
of $s$-index $1$ and $2$ respectively,
such that the chain recurrence classes $C(x_g,g)$ and $C(y_g,g)$ are trivial, that is, 
they are equal to the orbits of $x_g$ and $y_g$ respectively. 
Furthermore, we can require that the Hausdorff distance between 
the orbit of $x_g$ (resp. $y_g$) and the chain recurrence class of $C(p, f)$ is arbitrarily small. 
\end{theo}

In all the known examples of wild systems, as in \cite{BD1}, the process of creating 
a new class out of a given class consists of producing an attracting or a 
repelling region, for instance by changing the derivative of a periodic orbit, using Franks' Lemma.
In our context, the partially hyperbolic structure prevents 
the existence of attracting regions in ${\bf R}$: 
The unstable manifold of 
every periodic point in ${\bf R}=\bigcup R_i$ 
cuts transversely the stable manifold of our 
starting point $p$ (which has large stable manifold). 
If furthermore the  diffeomorphism is volume hyperbolic over $\bf R$ 
(i.e., contracts uniformly the area in $E^{cs}$), 
it prohibits the existence of repelling periodic orbits. 

Indeed, such examples do exist:
we also show that  there are examples of chain recurrence classes 
which are volume hyperbolic and simultaneously satisfy the hypothesis of the Theorem~\ref{t.wild}:
\begin{prop}\label{p.exa}
Given a compact 3-manifold $M$, there are non-empty $C^1$-open sets 
$\cU_{_{\mbox{\rm \tiny sdl}}}$, $\cU_{_{\mbox{\rm \tiny att}}}\subset\mathrm{Diff}^1(M)$
such that every $f \in\cU_{_{\mbox{\rm \tiny sdl}}}$  
(resp. $f \in\cU_{_{ \mbox{\rm \tiny att}}}$) admits 
a partially hyperbolic filtrating 
Markov partition ${\bf R}=\bigcup_{ i=1}^{k} R_i$ of saddle type  
(resp. of attracting type), periodic points 
$p, p_1$ and $q$ which satisfy the hypothesis 
of Theorem~\ref{t.wild} such that the partially hyperbolic splitting 
$E^{cs}\oplus E^u$ over the maximal invariant set in $\bf R$ 
(i.e., $\cap_{n\in \mathbb{Z}} f^n({\bf R}) $) is volume hyperbolic.
\end{prop}

Theorem~\ref{t.wild} announces the existence of 
an arbitrarily small perturbation of 
the initial diffeomorphism producing 
a saddle point with trivial chain recurrence 
class near the class of $p$. 
As we will see later, the hypothesis of Theorem~\ref{t.wild}
is $C^1$-open condition. The openness is clear except the 
persistence of partially hyperbolic filtrating Markov partitions,
which will be discussed in Section~\ref{s.statement}.
Thus, the production of new class can be done 
keeping the hypothesis.
Then, by repeating it infinitely many times, we have the following:

\begin{coro}\label{c.isolated} Under the hypothesis of Theorem~\ref{t.wild},  
there are a $C^1$-neighborhood $\cU$ of $f$ and a residual subset $\cR$ of $\cU$ such that, 
for every $g\in \cR$ the class $C(p_g,g)$ is the Hausdorff limits of two sequences of periodic orbits 
$\cO(x_i)$ and $\cO(y_i)$ of $s$-index $2$ and $1$ respectively and having trivial classes, that is,
$$C(x_i,g)=\cO(x_i) \mbox{ and } C(y_i,g)=\cO(y_i).$$ 

In particular, the chain recurrence class $C(p)$ is wild.
\end{coro}

The proof of Corollary~\ref{c.isolated} is done by Theorem~\ref{t.wild} and a standard argument 
involving Baire's category theorem, together with the upper semicontinuity of chain recurrence classes. 
We omit the proof.

As a direct corollary, we have the following:

\begin{coro}
On every closed $3$-manifold $M$ there is a non-empty $C^1$-open set $\cU\subset \mathrm{Diff}^1(M)$ 
such that every $f \in \mathcal{U}$ admits a  hyperbolic periodic point of $s$-index $2$ 
satisfying the following conditions:
\begin{itemize}
 \item the  chain recurrence class  $C(p)$ is wild 
 (and is a Hausdorff limit of a sequence of periodic orbits 
 with trivial classes);
 \item the class $C(p)$ is volume hyperbolic;
 \item the class $C(p)$ is not contained in a normally hyperbolic surface. 
\end{itemize}
\end{coro}

The argument involving the flexible periodic points in \cite{BS} 
genuinely depended on the invariance of $C^1$-distance under 
rescaling. 
Thus the direct application of techniques developed 
in this paper does not work under the $C^r$-topology for $r \geq 2$.
Consequently, we still do not have 
any answer to the following question.
\begin{ques} Given $r \in [2, +\infty]$ 
and a $3$-manifold $M$,  does there exist 
$C^r$-locally generic diffeomorphisms on $M$   with a wild and  volume hyperbolic chain recurrence class
$C(p)$ which is not contained in a normally hyperbolic surface? 
\end{ques}

\subsection{Generic diffeomorphisms without attractors or repellers}

Following \cite{BLY} (see \cite{BLY} for the bibliography of
the results stated below without citation), 
we say that a compact invariant set 
$\Lambda \subset M$ is a {\it topological attractor} if the following 
holds:
\begin{itemize}
\item it is transitive; 
\item it has a compact attracting neighborhood $U$ (i.e., 
$U$ is a compact neighborhood of $\Lambda$ satisfying 
$f(U) \subset \mathrm{int}(U)$) such that $\bigcap_{i=0}^{+\infty}f^i(U) = \Lambda$ holds.
\end{itemize} 
{\it Topological repellers} mean topological attractors for $f^{-1}$.
Recall that a \emph{quasi-attractor} is 
a chain recurrence class admitting 
a basis of attracting neighborhoods. 
By definition, an topological attractor is a quasi-attractor, 
but the converse is not true in general.

It is known that for $f \in \mathrm{Diff}^1(M)$,
quasi-attractors always exist. 
In dimension $2$, it is known that $C^1$-generic diffeomorphsism 
always have an attractor and a repeller.
In dimension larger than or equal to $4$, \cite{BLY} built 
$C^r$-generic diffeomorphisms, 
for any $r\geq 1$, without attractors or repellers. 
However, on $3$-manifolds the argument in \cite{BLY} 
was not enough to get generic diffeomorphisms 
without attractors or repellers: 
they only constructed $\cC^r$-generic diffeomorphisms without 
attractors but with (infinitely many) repellers.

Our example, together with
Theorem~\ref{t.wild}, allows to obtain 
generic existence of a quasi-attractor 
which is accumulated by saddles with trivial classes. 
As a result, we have the following:
\begin{coro}\label{c.wild1}
On every closed $3$-manifold $M$ 
there is a non-empty $C^1$-open set $\cU\subset \mathrm{Diff}^1(M)$ and a 
residual subset $\cR\subset \cU$ 
such that every $f\in \cR$ has 
no topological attractors or topological repellers. 
Moreover, such $\cU$ and $\cR$ can be chosen so that 
every $g\in \cR$ has finitely many wild 
volume hyperbolic quasi-attractors 
whose basins cover a residual subset of $M$.
\end{coro}

We give a brief account on the proof of Corollary~\ref{c.wild1} in section 4.

\medskip

Thus, for $C^1$-case, the problem of the existence of 
attractors and repellers is almost settled. 
However, again it seems that the direct application 
of this paper's  techniques do not work 
under $C^r$-topology for $r \geq 2$.
Therefore, the following is a remaining
problem in this line:
\begin{ques} Given $2\leq r\leq +\infty$ and 
a $3$-manifold $M$,  does there exist 
$C^r$-locally generic diffeomorphisms on $M$ 
without attractors or repellers? 
\end{ques}

\medskip

{\bf Organization of this paper} 
In section~\ref{s.statement}
we give the precise definition of partially hyperbolic 
filtrating Markov partitions and 
present some of the basic properties of them. In section 3, 
using these properties and 
the notion of $\varepsilon$-flexible points, we prove 
Theorem~\ref{t.wild}.
In section 4, we prove
Proposition~\ref{p.exa}: we construct examples of 
diffeomorphisms which satisfy the hypotheses of 
Theorem~\ref{t.wild} and the volume hyperbolicity.
We close Section 4 by giving a short explanation 
of Corollary~\ref{c.wild1}. 

\smallskip

{\bf Acknowledgement }We thank Dawei Yang, Ming Li, Sylvain 
Crovisier and Rafael Potrie for discussions on the notion of flexible points and their use for ejecting periodic points.

% !TEX root = yetwild.tex

\section{Partially hyperbolic filtrating Markov partitions}\label{s.statement}

In this section we give the precise definition of 
partially hyperbolic filtrating Markov partition and discuss several 
elementary properties of it. 
Throughout this article, $M$ denotes a closed, smooth, three-dimensional 
Riemannian manifold. We denote the group of $C^1$-diffeomorphisms of $M$
by $\mathrm{Diff}^1(M)$ and furnish it with the $C^1$-topology.

\subsection{Filtrating sets and chain recurrence classes}
Let us start with recalling the notion of filtrating set.
Let $f \in \mathrm{Diff}^1(M)$.
A compact set $A \subset M$ is 
called {\it attracting region} of $f$ if 
$f(A) \subset \mathrm{int}(A)$, where $\mathrm{int}(A)$
denotes the topological interior of $A$. 
A compact set $R \subset M$
is a {\it repelling region} of $f$ if it is an attracting region of $f^{-1}$. 
Finally, a \emph{filtrating set} $U$ of $f$ is a 
compact subset of $M$ 
which is an intersection of an attracting region $A$ and 
a repelling region $R$.
The main property of a filtrating set $U$ is that any point $x\in U$ which goes out $C$
(i.e., $f(x)\notin U$) has no future return to $U$: 
for every $n>0$ one has $f^n(x)\notin U$. 

Let us fix a distance function $d$ on $M$
and let $\varepsilon >0$. 
We say that a sequence of points $\{ x_n \}_{n = 0, \ldots, n_0}$ is 
an \emph{$\varepsilon$-pseudo orbit} of $f$
if for every $k \in [0, n_0 -1]$ we have $d(f(x_k), x_{k+1}) < \varepsilon$.
The \emph{chain recurrence set} $\cR(f)$  of $f$ is the set of points $x$ such that for every $\varepsilon>0$ 
there is an $\varepsilon$-pseudo orbit $x_0=x, x_1, \dots, x_{n_0}=x$ 
for some $n_0\geq 1$.
According to Conley theory, 
the chain recurrence set is naturally divided into a disjoint union of compact subsets called 
\emph{chain recurrence classes} 
by the equivalence relation on $R(f)$
defined as follows: two points $x,y\in\cR(f)$ 
belongs to the same chain recurrence class if 
for every $\varepsilon>0$ there are two $\varepsilon$-pseudo orbits, one starting from $x$ and 
ending at $y$ and the other starting from $y$ and ending at $x$. 
One can see that the notion of chain recurrence and chain recurrence 
classes do not depend on the choice of $d$.

It is easy to check that every chain recurrence class $C$ meeting a filtrating set $U$ is contained in $U$. 
On the other hand, Conley theory implies that every chain recurrence class admits a basis of 
neighborhoods which are filtrating sets.

Finally, note that being a attracting region is a $C^0$ (and therefore $C^1$)-robust property: if $A$ is an 
attracting region for $f$, then $A$ is an attracting region for every $g$ sufficiently $C^0$-close to $f$.  
Similarly, being a repelling region is a $C^0$-robust property
and consequently the same  holds for filtrating sets. 

\subsection{Invariant cone fields, 
partial hyperbolicity and volume hyperbolicity}

In this subsection we review the notion of cone fields and 
partial hyperbolicity.

Consider the cone $C=\{(x,y,z)\in\RR^3, |z|\geq \sqrt{x^2+y^2}\}$. 
In this paper a \emph{cone $\cC_x$} at a point $x \in M$ means 
the image of $C$ by some linear isomorphisms from $\RR^3$ to $T_xM$. 
A \emph{(continuous) cone field} $\cC=\{\cC_x\}$ on a subset $U\subset M$ is a 
family of cones $\cC_x$ of $T_xM$ ($x \in U$) which are locally defined 
as the image of $C$ by a linear isomorphisms $A_x\colon\RR^3\to T_xM$ 
depending continuously on $x$. 
In this article, we only treat continuous cone fields. 
Hence we omit the word continuous for cone fields.

For two cones $C_1$, $C_2$ at same point,  
we say that $C_1$ is \emph{strictly contained} in $C_2$ if 
$$C_1\subset \mathrm{int}(C_2)\cup\{0\}.$$ 
In the same way, we say that a linear subspace $L$ is \emph{strictly contained}  in a cone $\cC$ if 
$L\subset \mathrm{int}(\cC)\cup\{0\}$. 

\begin{defi} Let $U \subset M$.  
A cone field $\cC$ on $U$  is said to be \emph{strictly invariant} under $f$   
if for every $x\in U$ with $f(x)\in U$ one has 
that $Df_x(\cC_x)$ is strictly contained in $\cC_{f(x)}$. 

A cone field $\cC$ is said to be \emph{unstable cone field} 
if it is strictly invariant and there is a Riemann metric $\|\, \cdot \, \|$ on $M$ such that 
for every $x\in U$ with $f(x)\in U$, one has $\|Df(v)\|>\|v\|$
for every vector $v\in \cC_x\setminus\{0\}$ .
\end{defi}

The notion of unstable cone field is closely related to the notion of partial hyperbolicity.
\begin{defi}
 Let $K\subset M$ be an $f$-invariant compact set.  We say that $K$ is partially hyperbolic (with $1$-dimensional 
 unstable bundle) if there is a $Df$-invariant continuous splitting $TM|_{K}= E^{cs}\oplus E^u$ such that 
 the following holds:
 \begin{itemize}
  \item The splitting is \emph{dominated}: 
 there is $n>0$ such that for every $x \in K$,
 for every unit vector$u\in E^{cs}(x)$ and 
  $v\in E^u(x)$ we have the following inequality: 
  $$\|Df^n(u)\|<\frac12 \|Df^n(v)\|.$$
 By $E^{cs}\oplus_< E^{u}$ we mean that the 
 bundle $E^{cs}$ is dominated by $E^{u}$.
  \item The bundle $E^u$ is \emph{uniformly expanding}: there is $n>0$ such that for every $x \in K$ and for every unit vector 
  $v\in E^u(x)$, one has $\|Df^n(v)\|>1$.
 \end{itemize}

\end{defi}

In \cite{Go}, Gourmelon showed that, 
in the previous definition we can always 
 choose a Riemannian metric for which $n=1$. 
As a consequence we have the following: 

\begin{lemm}\label{l.conenbd}
 Let $U\subset M$ be a compact subset and 
$K=\bigcap_{n\in\ZZ} f^n(U)$ be its maximal invariant set. 
Then,

\begin{itemize}
 \item $K$ admits a dominated splitting $E\oplus_< F$  
 if and only if there is a strictly invariant cone field $\cC$ on some 
 neighborhood of $K$.
 \item $K$ is partially hyperbolic (with a $1$-dimensional unstable bundle) if and only if there is a
 strictly invariant unstable cone field $\cC$ on some neighborhood of $K$. 
\end{itemize}
\end{lemm}

Let $K$ be an invariant set with a dominated splitting $E\oplus_{<}F$
where $\dim E=2$ and $\dim F =1$.
We say that $K$ is {\it volume hyperbolic} if the determinant
 of the derivative  $Df$ restricted to $E$ is uniformly contracting and 
$Df$ restricted to $F$ is uniformly expanding.
In this paper, 
for a linear isomorphism $L\colon V_1\to V_2$  between two Euclidean 
spaces and a subspace $W_1\subset V_1$,
by \emph{determinant of the restriction $L|_{W_1}$} 
we mean
the determinant calculated with respect to the
Euclidean structure that $W_1$ and $L(W_1)$ inherit from $V_1$ and $V_2$, respectively.

%In particular, $K$ is partially hyperbolic with unstable bundle $F$.

The existence of a dominated, partially hyperbolic, or volume hyperbolic  splitting 
on a maximal invariant set is a $C^1$-robust property
in the following sense (see \cite{BDV} for the detail).
\begin{lemm}
Let $U$ be a compact subset of $M$, and 
$K=\bigcap_{n\in\ZZ} f^n(U)$ be its  maximal invariant set.
Suppose $K$ admits a dominated splitting $E\oplus_{<}F$. Then,
\begin{itemize}
 \item if $g$ is sufficiently $C^1$-close to $f$, then the maximal 
 invariant set $K_g :=  \bigcap_{n\in\ZZ} g^n(U)$ also 
 admits a ($g$-invariant) 
 dominated splitting $E_g \oplus_{<} F_g$, where 
 $E_g$ and $F_g$ are the continuations of $E$ and $F$, respectively.
 \item if the splitting $E\oplus_{<}F$ is volume hyperbolic and $g$ is 
 sufficiently $C^1$-close to $f$, then 
 $E_g \oplus_{<} F_g$ is also volume hyperbolic.
 \end{itemize}
 \end{lemm}

\subsection{Partially hyperbolic filtrating Markov partitions}\label{ss.partialhyp}
In this subsection, we give the precise definition of
partially hyperbolic filtrating Markov partitions. 

A compact subset $R$ of $M$ is said to be a \emph{rectangle} if it is 
$C^1$-diffeomorphic to 
the full compact cylinder $\DD^2\times [0,1]\subset \RR^3$, where $\DD^2$ 
is the unit disc of $\RR^2$.  
We endow $\DD^2\times [0,1]$ with the coordinates 
$(x,y,z)$ of $\RR^3$. 

We denote by $\partial_s R$ the image of 
$(\partial \DD^2)\times [0,1]$ (which is diffeomorphic to the 
annulus $S^1 \times [0,1]$) and 
call it the \emph{side boundary} of $R$. 
Also, we denote by $\partial_l R$ the image of 
$\DD^2\times  \partial [0,1] =\DD^2\times  \{0,1\}$ 
and call it the \emph{lid boundary} of $R$
(which has two connected components).   
Given a rectangle $R \subset M$, 
a \emph{vertical sub-rectangle} of $R$ is a rectangle $R_1\subset R$ such that
the following holds:
\begin{itemize}
 \item $R_1$ is disjoint from $\partial_s(R)$, and
 \item  $\partial_l (R_1)\subset \partial_l(R)$ and each connected component of $\partial_l(R)$ 
 contains exactly one connected component of $\partial_l(R_1)$.
\end{itemize}
A \emph{horizontal sub-rectangle} $R_2$ of $R$ is a rectangle $R_2\subset R$ satisfying the following:
\begin{itemize}
 \item $\partial_s(R_2)\subset \partial_s(R)$ and $\partial_s(R_2)$ is a essential sub-annulus of the annulus 
 $\partial_s(R)$, and
 \item each connected component of $\partial_l(R_2)$ is either disjoint from $\partial_l(R)$ or 
 coincides with a connected component of $\partial_l(R)$. 
\end{itemize}

A cone field $\cC$ on $R$ is called \emph{vertical} if there is a diffeomorphism 
$\varphi\colon R\to \DD^2\times [0,1]$ such that for every $p\in R$ the cone
$D\varphi(\cC(p))$ contains the vertical vector $\frac\partial{\partial z}$ and is 
transverse to the horizontal plane field spanned by 
$\frac\partial{\partial x}$ and $\frac\partial{\partial y}$.

If $\cC$ is a vertical cone field on $R$ and $R_1\subset R$ is a vertical 
sub-rectangle, we say that $R_1$ is 
\emph{$\cC$-vertical} if the restriction of $\cC$ to $R_1$ is a vertical cone field of $R_1$.
Also, if $\cC$ is a vertical cone field on $R$ and $R_1\subset R$ is a horizontal sub-rectangle, 
we say that $R_1$ is \emph{$\cC$-horizontal} if the restriction 
of $\cC$ to $R_1$ is a vertical cone field of $R_1$.

\begin{defi}\label{d.saddle}
Let $f \in \mathrm{Diff}^1(M)$. A compact subset ${\bf R}$ of  $M$ is said to be 
a \emph{partially hyperbolic filtrating Markov partition of saddle type of $f$} if:
\begin{itemize}
\item ${\bf R}$ is a filtrating set: ${\bf R} = A \cap R$ where $A$ is an attracting region and 
$R$ is a repelling region.
\item ${\bf R}$ is the union of finitely many pairwise disjoint rectangles 
${\bf R} = \bigcup_{i=1,\dots, k} R_i$.
\item The side boundary $\partial_s(R_i)$ is contained in $\partial A$ and the lid boundary 
$\partial_l(R_i)$  is contained in $\partial R$ (the boundary of repelling region).
\item For every $(i, j)$, each connected component of $f(R_j)\cap R_i$ is
 a vertical sub-rectangle of  $R_i$.
\item There is a strictly invariant unstable cone field $\cC^u$ on ${\bf R}$ which is vertical on each $R_i$. 
\end{itemize}
\end{defi}

\begin{defi}\label{d.attracting}
Let $f \in \mathrm{Diff}^1(M)$. A compact subset ${\bf R}$ of  $M$ is said to be 
a \emph{partially hyperbolic filtrating Markov partition of attracting type of $f$} if:
\begin{itemize}
\item {\bf R} is an attracting region.
\item ${\bf R}$ is the union of finitely many rectangles 
${\bf R} = \bigcup_{i=1,\dots, k} R_i$.
\item For every $i\neq j$, one of the following holds:
\begin{itemize}
 \item either $R_i$ and $R_j$ are disjoint,
 \item or, $R_i\cap R_j$ is exactly one connected component of $\partial_l(R_i)$ or of $\partial_l(R_j)$ 
 contained in the interior of a component of  $\partial_l(R_j)$ or of $\partial_l(R_i)$ respectively.
 In this case we say that $R_i$ and $R_j$ are \emph{adjacent}. 
\end{itemize}
\item The side boundary $\partial_s(R_i)$ is contained in $\partial A$. 
\item For every $(i, j)$,  
each connected component of $f(R_j)\cap R_i$ satisfies one of the following conditions:
\begin{itemize}
\item Either it is a vertical sub-rectangle of $R_i$,
\item or, it is a connected
component of $f(\partial_l (R_j))$ contained in a component of $\partial_l (R_i)$.
\end{itemize}
\item There is a strictly invariant unstable cone field $\cC^u$ on ${\bf R}$ which is vertical on each $R_i$. 
\end{itemize}
\end{defi}

\begin{rema}
In Definition~\ref{d.saddle}, the fourth condition 
(on the shape of each connected 
components) can be derived from the other conditions. Meanwhile, 
the corresponding condition 
in Definition~\ref{d.attracting} cannot. 
Hence, for the sake of the consistency, 
we included this condition into the definition.
\end{rema}

\begin{rema}
As the cone field $\cC$ is strictly invariant, one can show that, 
in Definitions~\ref{d.saddle} and \ref{d.attracting}, 
if a connected component of $f(R_j)\cap R_i$ is a 
vertical sub-rectangle of $R_i$, then it is also  $\cC$-vertical. 
\end{rema}

It would be possible that one gives a more general, conceptual definition of 
partially hyperbolic filtrating Markov partition 
including both types, but the definition and some proofs would become more technical.  
Since the main purpose of this article is to present  examples, 
we prefer to stop pursuing the generality and continue the study of them.

In the rest of this section, we will investigate the basic properties 
of partially hyperbolic filtrating Markov partitions. 
Most of the proofs are quite classical. Hence we often avoid the formal proofs and only give the 
sketch of them.
In many cases, the same proof works for the saddle type and the attracting type. 
However, the following lemma requires some modification depending on the types.

\begin{lemm}\label{l.continuation}
The property of being a partially hyperbolic filtrating Markov partition 
(of saddle or attracting type) is a $C^1$-robust property in the following sense:
\begin{enumerate}
\renewcommand{\labelenumi}{(\arabic{enumi})}
 \item Assume that ${\bf R}=\bigcup R_i$ is a partially hyperbolic filtrating Markov partition of saddle type 
 of a diffeomorphism $f$. 
 Then there is a $C^1$-neighborhood $\cU$ of $f$  such that for any $g\in\cU$, 
 $\bf R$ is a partially hyperbolic filtrating Markov partition of saddle type as well.
 \item Assume now that ${\bf R}=\bigcup R_i$ is a partially hyperbolic filtrating Markov partition of 
 attracting type 
 of a diffeomorphism $f$. Then there is a $C^1$-neighborhood $\cU$ of $f$  so that, for any $g\in\cU$, 
 there is a diffeomorphism $\psi$ which is $C^1$-close to the identity such that ${\bf R}_g = \bigcup \psi(R_i)$ 
 is a partially hyperbolic filtrating Markov partition of saddle type for $g$. 
\end{enumerate}

\end{lemm}
\begin{proof}[Sketch of the proof]
Let us start the proof for the saddle type case.
Being a filtrating  set is a $C^0$-robust property. 
Having a strictly invariant unstable cone field is a $C^1$-robust property 
and the fact that a cone field is vertical  over a cylinder is also $C^1$-robust. 
Therefore the unique difficulty is to obtain the condition on the intersection $f(R_j)\cap R_i$.

The existence of the invariant cone field implies that $f(\partial_s(R_j))$ 
cuts transversely $\partial_l(R_i)$.  
Notice that  $\partial_s(R_j)$ and $\partial_l(R_i)$ are compact surfaces with boundary 
and the fact that $f(\partial_s(R_j))\subset \partial A$ and $\partial_l(R_i)\subset \partial R$
implies that this intersection $f(R_j)\cap R_i$ is disjoint from 
$\partial_s(R_i)$. 
Now one can deduce that, for $g$ which is sufficiently $C^1$-close to $f$,
each connected component of $g(R_j)\cap R_i$ remains to be a $\cC$-vertical sub-rectangle,
since the side boundary of each connected component of $g(R_j)\cap R_i$ varies $C^1$-continuously. 

For the attracting type,  the union $S_f=\bigcup_i \partial_l(R_i)$ is a 
strictly positively $f$-invariant compact surface with boundary.  A priori, this surface is not invariant
for $g$ close to $f$. So we need to find a candidate for the lid boundaries of rectangles for $g$. 

Note that the existence of the 
vertical strictly invariant unstable cone field $\cC$ implies that $S_f$ is normally hyperbolic. 
Therefore, the persistence of normally hyperbolic invariant 
manifolds (see \cite{HPS} for example)  
implies that $\partial_l R_i$ varies continuously under $C^1$-small perturbations of $f$. 

Thus for $g$ which is $C^1$-close to $f$, we can construct cylinders $R_{i,g}$ whose boundary is 
close to the boundary of corresponding $R_i$ and so that 
$S_g=\bigcup{\partial_l(R_{i,g})}$ is a positively strictly invariant compact surface with boundary. 
Then, take the union $R_g=\bigcup R_{i,g}$. 
One can deduce that $R_g$ is an attracting region for $g$ and 
$ \bigcup R_{i,g}$ is an attracting partially hyperbolic Markov partition for $g$. 
\end{proof}

\subsection{Iteration, refinement of Markov partitions}

Let ${\bf R}=\bigcup R_i$ be a partially hyperbolic filtrating Markov partition 
(of saddle or attracting type) of $f$
endowed with a vertical unstable cone field $\cC$. 
Next proposition follows from standard arguments 
in hyperbolic dynamics, together with 
strongly invariant vertical unstable cones.

\begin{prop}\label{p.image_easy}
\begin{enumerate}
\renewcommand{\labelenumi}{(\arabic{enumi})}
 \item Suppose that the connected component $f(R_i) \cap R_j$ 
 is a vertical sub-rectangle of $R_j$. Then it is a 
 $\cC$-vertical and $Df(\cC)$-vertical sub-rectangle of $R_j$. 
\item Similarly, if the connected component of $f^{-1}(R_i) \cap R_j$
is a horizontal sub-rectangle of $R_j$, then it indeed is 
both $\cC$-horizontal and $Df^{-1}(\cC)$-horizontal sub-rectangle of $R_j$.
\end{enumerate} 
\end{prop}

\begin{rema}
The case where the connected component of $f(R_i) \cap R_j$
fails to be a vertical sub-rectangle happens only for attracting type. 
In such a case, 
the connected component is a $C^1$-disc contained in the 
interior of $\partial_l R_j$.
\end{rema}

\begin{proof}[Proof of Proposition~\ref{p.image_easy}]
The proof of item (1) and (2) are similar so we only give the proof 
of the first one. Suppose that $R'$ is a connected component of 
$f(R_i) \cap R_j$ which is a vertical sub-rectangle of $R_j$.
We need to find  coordinates on $R'$ which guarantees that
$R'$ is a $\cC$-vertical sub-rectangle or $Df(\cC)$-vertical sub-rectangle.
Let us find such coordinates. 

Since $R_i$ is a $\cC$-vertical rectangle, 
we can a the one-dimensional foliation $\cF_i$ tangent to $\cC$
on $R_i$ by taking the image of the $z$-direction.
Also, since $R_j$ is a $\cC$-vertical rectangle, 
we can take the two-dimensional foliation $\cH_j$ which is transverse to $\cC$
on $R_j$ by taking the image of the $xy$-plane.

Then, we consider the foliation $f(\cF_i)$ and $\cH_j$ on $R'$.
Since $\cC$ is $f$-invariant, we can see that 
$f(\cF_i)$ and $\cH_j$ are transverse at each point. 
Now we take a coordinate which sends $f(\cF_i)$ parallel to the $z$-direction,
$\cH_j$ parallel to the $xy$-direction and $R'$ to 
$\mathbb{D}^2 \times I$. In this coordinates, one can check that 
$Df(\cC)$ and (in particular) $\cC$ contain the $z$-direction and $\cC$ 
(and therefore $Df(\cC)$) does not 
contain any vector in the $xy$-direction.
This concludes that the sub-rectangle $R'$ is both $\mathcal{C}$- and $Df(\cC)$-vertical. 
\end{proof}

By a similar argument and taking the attracting case into consideration, 
we can prove the following:
\begin{prop}\label{p.image}
\begin{enumerate}
\renewcommand{\labelenumi}{(\arabic{enumi})}
 \item Let $\tilde R_i\subset R_i$ be a $\cC$-vertical sub-rectangle.  
 Then for every $j$, each 
 connected component $\Ga$ of 
 $f(R_i)\cap R_j$ contains exactly one connected 
 component  $\tilde \Ga$ of $f(\tilde R_i)\cap R_j$. 
 
 Notice that $\Ga$ is either a $\cC$-vertical sub-rectangle 
of $R_j$ or is contained in the interior of the 
 lid boundary $ \partial_l(R_i)$. 
In the first case,  $\tilde \Ga$ is $\cC$- and $Df(\cC)$-vertical sub-rectangle of $R_j$.  
% Furthermore, $\tilde \Ga$ is a vertical rectangle for the cone field.
%  \item  Every connected component of $f^{-1}(C_j)\cap C_i$ is 
% \begin{itemize}
%  \item either a horizontal subrectangle of $C_i$
%  \item or (in the attracting case) a connected component of $\partial_l(C_i)$. 
%\end{itemize}
\item Let $\hat R_j\subset R_j$ be a $\cC$-horizontal sub-rectangle.  Then for every $i$, each 
connected component 
$\De$ of $f^{-1}(R_j)\cap R_i$ contains exactly one connected component  $\hat \De$ of 
$f^{-1}(\tilde R_j)\cap R_i$.  Furthermore, if $\De$ is a horizontal sub-rectangle of 
$R_j$ then $\hat \De$ is a $\cC$- and $Df^{-1}(\cC)$-horizontal sub-rectangle of $R_j$. 
% Otherwise, $\hat \De$ is a component of $\partial_l(R_i)$
%(this case happens only for attracting type) . 
\end{enumerate} 
\end{prop}
We omit the proof of Proposition~\ref{p.image}.

\medskip

If ${\bf R} = \bigcup R_i$  is a partially hyperbolic filtrating Markov partition of $f$, we denote by 
$\{ {\bf R}\cap f({\bf R}) \}$ the family of connected components of $f(R_i)\cap R_j$
which are vertical sub-rectangles (i.e., those which have non-empty interior). 

As a direct corollary of Proposition~\ref{p.image}, we have the following:
\begin{coro}
Let ${\bf R} = \bigcup R_i$ be a partially hyperbolic filtrating Markov partition  endowed with the 
vertical unstable cone field $\cC$. 
Then ${\bf R}\cap f({\bf R})$ is a partially hyperbolic filtrating Markov partition (of the same type) for the family 
of rectangles $\{ {\bf R}\cap f({\bf R})\}$.
The cone fields $\cC$ and $Df(\cC)$ are strictly invariant
vertical unstable cone field on ${\bf R}\cap f({\bf R})$.

Similarly, ${\bf R}\cap f^{-1} ({\bf R})$ is a partially hyperbolic filtrating Markov partition (of the same type) 
with the family of rectangles $\{ {\bf {\bf R}}\cap f^{-1}({\bf {\bf R}}) \}$.
The cone fields $\cC$ and $Df^{-1}(\cC)$ are strictly invariant, vertical, unstable cone field on it.
\end{coro}

Iterating this process, we also have the following.

\begin{coro} For any integers $m\leq n$, the intersection $\bigcap_{i=m}^n f^i({\bf R})$ 
is a partially filtrating Markov partition with the family of 
rectangles $\{\bigcap_{i=m}^n f^i({\bf R})\}$.
The cone field $Df^i(\cC)\,$ is strictly invariant, vertical, unstable cone field on it for every $i$ satisfying $m\leq i\leq n$. 
\end{coro}

For any integers $m\leq n$, 
the partially hyperbolic filtrating Markov partition 
$\bigcap_{i=m}^n f^i({\bf R})$ is called a \emph{refinement} of ${\bf R}$.

\subsection{Center-stable discs and unstable lamination}

For a partially hyperbolic filtrating Markov partition 
${\bf R} =\bigcup R_i$ of $f$, we can associate two important 
dynamically defined sets, namely, the {\it unstable lamination} and
the {\it center stable discs}.
These sets have certain invariance under the 
iteration of $f$.

\begin{lemm}
\begin{itemize}
 \item The positive maximal invariant set $\La_+=\bigcap_{n\geq 0} f^n({\bf R})$ is 
a $1$-dimensional lamination whose leaves intersect each 
rectangle $R_i$ as a continuous family of $C^1$-segments 
whose tangent space is contained in the cone field $Df^n(\cC^u)$ at 
every point and for every $n\geq 0$. 

\item For the saddle type:  the negative maximal invariant set $\La_-=\bigcap_{n\leq 0} f^n({\bf R})$
is a $2$-dimensional lamination 
whose leaves are a continuous family of $C^1$-discs contained in the rectangles $R_i$, with boundary 
contained in
$\partial_s(R_i)$ and transverse to each cone field $Df^n(\cC^u)$. 
These discs are called \emph{the center-stable discs}. 

\item For the attracting  type: the partially hyperbolic filtrating 
Markov partition ${\bf R}=\bigcup R_i$ admits a 
unique positively invariant two-dimensional foliation satisfying the following properties:
\begin{itemize}
 \item each leaf is a $2$-disc transverse to each cone field $Df^n(\cC^u)$,
 \item the connected components of $\bigcup_k\partial_l(R_k)$ are leaves of this foliation,
 \item the family of discs induces on each $R_i$ a continuous family of $C^1$-discs,
 \item the family of discs is upper semi continuous for the Hausdorff distance, the discontinuity being 
 the $\partial_l R_k$. 
\end{itemize}
These discs are called \emph{the center-stable discs}.

\end{itemize}
\end{lemm}
\begin{proof}[Sketch of the proof]
The existence of the one-dimensional unstable lamination is classical in 
partial hyperbolic dynamic. Let us build the center-stable discs which are the leaves of the center-stable 
foliation/lamination. For the sake of simplicity we only consider saddle type.

Let us fix one connected component $\Delta$ of $\bigcap_{n\leq 0} f^n({\bf R})$.
Using Proposition~\ref{p.image} repeatedly, we see that it is given as the 
limit of nested sequence of horizontal sub-rectangles $\{\tilde{R}_l\}$
where $\tilde{R}_l$ is the (unique) rectangle of the negative 
refinements $\bigcap_{i=-l}^0 f^i({\bf R})$ which contains $\Delta$.

As $f^{-1}$ contracts the vectors in the cone $Df(\cC^u)$, 
one can check that the thickness of 
the rectangle tends to $0$ as $l \to +\infty$. 
On the other hand, (the tangent space of) their lid 
boundary is contained in the complement cone of 
$Df^{-l}(\cC^u)$, and this complement cone tends to a plane.  
Hence, we can see that such a decreasing sequence of such rectangles
tends to a $C^1$-disc. 
The continuity of the family of discs follows by construction.
\end{proof}

Using the notion of center-stable discs, we can define the \emph{largeness} of 
stable manifolds of hyperbolic periodic points (of $s$-index $2$)
in a partially hyperbolic filtrating Markov partition.

\begin{defi}
We say that a hyperbolic periodic point $x\in R_i$ of $s$-index $2$  
has {\it large stable manifold} 
if the center-stable disc through $x$ is contained in $W^s(x)$.
\end{defi}
Notice that for a hyperbolic periodic point of $s$-index $2$ in ${\bf R}$,
having large stable manifold is a $C^1$-robust property.

The next lemma shows that the property of having 
large stable manifold is a property of the orbit.

\begin{lemm}\label{l.large} If $x \in {\bf R}$ 
is an $s$-index $2$ hyperbolic 
periodic point with large stable manifold, then every point $f^i(x)$ 
in the orbit of $x$ has large stable manifold.
\end{lemm}
\begin{proof} Just notice that, for any center-stable disc $D$, each connected component of  
$f^{-1}(D)\cap {\bf R}$ is a center-stable disc. 
Then we can see that the stable disc through 
$f^{-1}(x)$ is contained in $W^s(f^{-1}(x))$. 
Now we can get the conclusion by a simple inductive argument. 
\end{proof}

The following lemma says that, in a hyperbolic basic set,  
the largeness of a periodic point is inherited to the periodic 
points passing nearby.

\begin{lemm}\label{l.large2}  Let $K\subset {\bf R}$ 
be a hyperbolic basic set of $s$-index $2$. 
Assume that $x\in K$ is a periodic point with large stable manifold.
Then, there is a neighborhood $U$ of $x$ and a $C^1$-neighborhood 
$\cU$ of $f$ such that the following holds:
for every $g\in\cU$ and every periodic 
point $y\in K_g$ whose orbit $\cO(y,g)$ meets $U$, 
the stable manifold through $y$ is large. 
\end{lemm}
\begin{proof}
 The lemma follows from the fact that local stable manifolds of points in a hyperbolic 
 set vary continuously with the point and with the diffeomorphism.
 Therefore, for $g$ close to $f$ and $y\in K_g$ close enough to $x$, 
 the stable manifold $W^s(y,g)$ 
 contains the whole  center-stable disc which passes through $y$. 
 Since the property of having large stable manifold 
 is invariant under the iteration 
 (for periodic points), the property that the orbit 
 of $y$ passes close enough to $x$ ensures the largeness of 
 the stable manifold through $y$. 
\end{proof}

\subsection{Topology of Markov partitions in the large stable manifolds and 
expulsion of periodic points from chain recurrence classes}
Now we are ready to describe the ``topology of a chain recurrence class''
mentioned in the  introduction. 
It plays an important role to determine the 
possibility of escaping from a chain recurrence class.

In this subsection, $D$ denotes a periodic center-stable disc of \emph{period} $\pi(D)$
(that is,  $D$ satisfies $f^{\pi(D)}(D) \subset D$
and $\pi(D)$ is the smallest positive integer satisfying this condition)
in a partially hyperbolic filtrating Markov partition ${\bf R}=\bigcup R_i$
such that the first return map $f^{\pi(D)}|_{D}$ preserves the orientation. 
By $\lambda_D$  we denote the 
forward maximal invariant set in $D$, that is, we put
$\lambda_D :=\bigcap_{n\geq 0} f^{n\pi(D)}(D)$.

\begin{lemm}\label{l.TD}
Under the above setting, consider the orbit space 
of $D \setminus \lambda_D$. 
Then it is naturally diffeomorphic to the two dimensional torus. 
We denote it by $\TT_D$.
\end{lemm}

\begin{proof}[Sketch of the proof]
One can see this orbit space as follows: consider 
a fundamental domain of $f^{\pi(D)}$ which is a 
annulus bounded by a smooth curve $\gamma$ and $f^{\pi(D)}(\gamma)$. 
Then the quotient space is obtained from this annulus
by gluing $\gamma$ to $f^{\pi(D)}(\gamma)$ along $f^{\pi(D)}$.  
\end{proof} 

In this orbit space $\TT_D$, there is a special (homotopy class of a) curve.

\begin{defi}\label{d.para}
The closed curve $\partial D$ induces a well-defined
homotopy class in $\pi_1(\TT_D)$.
We call it the \emph{parallel} of $\TT_D$. 
\end{defi}
Note that any curve in $D\setminus \lambda_D$ joining a point
$y \in (D\setminus \lambda_D) $and $f^{\pi(D)}(y)$  projects to a 
closed curve in $\TT_D$  whose (topological) intersection number with 
the parallel is equal to $1$.

\begin{lemm}\label{l.obstacle}
Under the above setting,
$D\cap f^{\pi(D)}({\bf R})$ consists of the union of finitely many 
pairwise disjoint discs, one of them being $f^{\pi(D)}(D)$.
\end{lemm}

\begin{proof}[Sketch of the Proof]
Let $R_D$ be  a rectangle which contains $D$.
For the saddle type $R_D$ is uniquely determined. For the 
attracting case there may be two such rectangles. We  choose one of them.
Then, by applying Proposition~\ref{p.image} repeatedly,
we can see that each connected component of
$f^{\pi(D)}({\bf R})\cap R_D$ is a $\cC$-vertical sub-rectangle of $R_D$
or a disc contained in a lid boundary.
As a result, we can see that the number of connected components of
$D\cap f^{\pi(D)}({\bf R})$ is finite and
that each connected component is a $C^1$-disc in $D$.
\end{proof}

We denote by $\De_D$ the projection of 
$\left(D\cap f^{\pi(D)}({\bf R})\right)\setminus f^{\pi(D)}(D)$
to $\TT_D$.  
It consists of finitely many pairwise disjoint discs embedded in the torus $\TT_D$. 
Suppose that we have a hyperbolic periodic point $x \in {\bf R}$ 
of $s$-index $1$. 
We are interested in deciding if the chain recurrence class of $x$ is trivial or not. 
To see this, it is important to compare the set $\De_D$ and the projection of 
the stable manifold of $x$ in $\TT_D$.

First, let us have the following observation:

\begin{rema}
Let $x\in D$ be a hyperbolic periodic point of period $\pi(D)=\pi(x)$ and with $s$-index $1$.   
% Then we have $W^s(x)\subset W^s(D)$. 
Assume that $W^s(x)\cap \lambda_D=\{x\}$.  
Then for the punctured stable manifold
$W^s(x)\setminus \{x\}$ projected on $\TT_D$ 
we have one of the following:
\begin{itemize}
 \item Either it is a union of two 
 disjoint simple closed curves such that 
 for each curve the intersection number 
 with the parallel is equal to $1$
 (this case happens when the stable eigenvalue of $x$ is positive);
 \item or, one simple closed curve whose 
 intersection number with the parallel is $2$
  (this case happens when the stable eigenvalue of $x$ is negative). 

\end{itemize}
\end{rema}

% \begin{lemm}\label{l.escaping}
% Let $\tilde \gamma\subset \TT_D$ be a closed simple curve  having an algebraic  intersection number
% with the parallel equal to $1$, and disjoint from $\De_x$. 
% 
% Then $\tilde\gamma$ lifts on $D\setminus \lambda_D$ in an $f^{\pi(D)}$-invariant simple 
% curve disjoint from the 
% positively maximal invariant set $\La^+=\bigcap_{n\geq 0} f^n(C)$. 
%\end{lemm}
%
%\subsection{Ejectable periodic points}
%
%Let $D$ be a periodic disc of period $\pi(D)$. We denote by $W^s(D)=\bigcup_{n>0} f^{-n\pi(D)}(D)$. Notice that 
% the orbit space of $f^{\pi(D)}$ in $W^s(D)\setminus \lambda_D$ is precisely $\TT_D$. 

Lemma~\ref{l.obstacle} shows the way for ensuring the triviality of 
a chain recurrence class of a periodic point in ${\bf R}$. 
 \begin{prop}\label{p.obstacle}  
Assume that $x\in D$ is  
an $s$-index $1$ hyperbolic periodic point 
with $\pi(x)=\pi(D)$ such that the projection 
of  
$W^s(x)\setminus \{x\}$ to $\TT_D$ is a disjoint union
of two simple closed curves having empty intersection with $\De_D$. 
Then the chain recurrence class of $x$ is trivial (i.e.,
 is equal to the orbit of $x$). 
\end{prop}
\begin{proof} 
We fix an attracting region $A$ and a repelling region $R$ 
such that ${\bf R}= A \cap R$.
Fix a disjoint union of
two compact intervals 
$I=I_1\cup I_2$ in
$W^s(x)\setminus \{x\}$ which contains a fundamental domain
for $f^{\pi(x)}$. It consists 
of two segments contained in different connected components of
$W^s(x)\setminus \{x\}$.

\begin{clai}~\label{c.funda}
There is an integer $i$ such that $f^i(I)\cap A=\emptyset$. 
\end{clai}
\begin{proof} Since $A$ is an attracting region,
for any $y\in M$, if  $f^i(y)\notin A$ then 
$f^j(x)\notin A$ for every integers $i$ and $j$ satisfying 
$j<i$.  By the compactness of $A$ and $I$, 
for proving the claim it is enough to prove that every
$y\in I$ has an itinerary out of $A$. 

For any point $y\in I$ 
we have an integer $j$ such that 
$f^j(y) \in D\setminus f^{\pi(D)}(D)$. 
By assumption, $f^j(y)$ does not belong to
$f^{\pi(D)}({\bf R})$. Meanwhile, since 
$x$ belongs to the interior of the repelling region 
$R$,  $W^s(x)$ is contained in $R$. 
Now one deduces that $f^j(y)\notin f^{\pi(D)}(A)$. 
Accordingly we have $f^i(y)\notin A$ for $i\leq j-\pi(D)$.
\end{proof}
%Since $W^s(x)\cap (D \setminus f^{\pi(D)}(D))$ is disjoint from $\De_D$ and
%each connected component of $f^{\pi(D)}(C)\cap C_D$ is $\mathcal{C}$-vertical,
By the Claim~\ref{c.funda} above, 
one can find a neighborhood $K$ of $I$ such that 
$f^i(K)$ is disjoint from $A$.
As $x$ is a periodic point 
contained in ${\bf R}$, 
we have $x \in f^n(A)\subset \mathrm{int}(A)$ for every $n >0$.

Suppose there exists a point $z \in \mathcal{R}(f) 
\setminus \mathcal{O}(x)$
(remember that $\mathcal{R}(f)$ denotes the chain recurrence set of $f$). 
By definition, for every $\varepsilon >0$ we can find an $\varepsilon$-pseudo
orbit $(y_k)_{k=0,\ldots, n}$ such that $y_0 = y_n = x$ and $y_{l} = z$ for some $l$.
Now, by using the fact that $x$ is a hyperbolic periodic point and $K$
contains a fundamental domain of $W^s(x)$ in its interior, 
one can deduce that if $\varepsilon$ is sufficiently small 
then the pseudo orbit $(y_k)$ must pass through $f^i(K)$.

However, notice that $K$ is disjoint from $A$.
This contradicts the assumption 
that $y_0 =x \in f^{\pi(D)}(A)\subset \mathrm{int}(A)$: 
for $\varepsilon$ small enough an $\varepsilon$-pseudo 
orbit starting at $x$ cannot go outside $A$. 
\end{proof}

\subsection{Mixing Markov partitions and uniqueness of quasi-attractor}

Given a partially hyperbolic filtrating Markov partition 
${\bf R} =\bigcup R_i$ of $f$ we associate an 
\emph{incidence matrix} which 
is a square matrix $A = (a_{ij})$ with integer entries, 
such that the term $a_{ij}$ is the number of connected components of $f(R_i)\cap R_j$ with non-empty interior. 

\begin{defi}
We say that a partially hyperbolic filtrating Markov partition 
${\bf R} = \bigcup R_i$ is \emph{transitive} 
if for every $(i, j)$ there exists an integer $n = n_{(i, j)}>0$ such that 
$(i, j)$-component of $A^n$ is non-zero.
Also, we say ${\bf R}$ is \emph{mixing} if there exists $n>0$
such that all the entries of $A^n$ are non-zero.
\end{defi}

We are interested in the transitivity or mixing property in regard to the 
uniqueness of the quasi-attractor in ${\bf R}$ for the attracting type. 
In this subsection, a \emph{segment} $\sigma \subset M$ means an image of an embedding of an interval in $M$.
A segment $\sigma \subset {\bf R}$ is \emph{$\mathcal{C}$-vertical} if for every point $s \in \sigma$
we have $T_{s}\sigma \subset \mathcal{C}_s$,
where $\mathcal{C}$ is a cone field on ${\bf R}$. 
A \emph{complete vertical segment} is a segment contained
in a rectangle $R_i$ such that the two end points connects the two different connected components 
of the lid boundary of $R_i$.

Let $\mathcal{C}$ be a strictly invariant unstable cone field over ${\bf R}$. 
If $\sigma$ is a $\mathcal{C}$-vertical segment in ${\bf R}$ satisfying $f(\sigma) \subset {\bf R}$,
then by using the uniform expansion property of $f$ on $\mathcal{C}$, 
we can see that the length of $f(\sigma)$
is greater than 
that of $\sigma$ multiplied by some constant strictly greater than $1$. 
Using this property, one can deduce the following: 
\begin{lemm}\label{l.segme}
Assume that ${\bf R} =\bigcup R_i $ is a 
mixing partially hyperbolic filtrating Markov partition of attracting type.
Let $\cC$ be a strictly invariant unstable cone field over it. 
Then given a segment $\sigma$
which is $\cC$-vertical, there is $n_0>0$ such that   
for any $n\geq n_0$, the set
$f^n(\sigma)\cap R_i$ contains 
a complete vertical  segment  for every $i$. 
\end{lemm}
Remember that a \emph{homoclinic class} of a hyperbolic periodic point is the closure 
of the set of points of transversal intersection between the stable and the unstable 
manifold of it.
As a direct corollary of Lemma~\ref{l.segme}, we have the following.
\begin{coro}\label{c.cut}
Suppose ${\bf R}$ is of attracting type and mixing, and $x \in {\bf R}$ is a hyperbolic periodic 
point of $s$-index $2$ with large stable manifold. Then,
\begin{itemize}
 \item The homoclinic class of $x$ is non-trivial.
 \item Every leaf of the unstable lamination of ${\bf R}$ cuts $W^s(x)$.
\end{itemize}
\end{coro}

%According 
%to \cite{BC},  for $C^1$-generic diffeomorphisms the $\omega$-limit sets of 
%generic points are quasi attractors.
Finally, we obtain the following:
\begin{coro}\label{c.quasi} Suppose that ${\bf R} = \bigcup R_i$ is a  mixing partially
hyperbolic filtrating Markov partition of attracting type and
$x\in {\bf R}$ is a periodic point with large stable manifold. 
Then for every $g$ which is sufficiently $C^1$-close to $f$, 
$g$ admits a unique quasi-attractor in ${\bf R}$ and it is equal to
the chain recurrence class of the continuation $x_g$ of $x$. 
Furthermore, the basin of this quasi-attractor is residual in ${\bf R}$. 
\end{coro}
\begin{proof} For the first claim, we just 
need to repeat the  argument in \cite{BLY}: 
any quasi-attractor contained in ${\bf R}$ is saturated by the strong unstable lamination.  Since each leaf of
the unstable lamination cuts the large stable manifold of $x$
(see Corollary~\ref{c.cut}), 
we deduce that $x$ belongs to the 
closure of the quasi-attractor, 
so it belongs to the quasi-attractor.  This shows that 
every quasi-attractor in ${\bf R}$ 
is the chain recurrence class of $x$,
which implies its uniqueness. Since the largeness of the stable manifold
for a periodic point and the existence of partially hyperbolic 
filtrating Markov partitions are both $C^1$-robust (see Lemma~\ref{l.continuation} for the second item),
we have the uniqueness for $g$ sufficiently $C^1$-close to $f$. 

Now let $U_n$ be a basis of neighborhood of the class of $x$ consisting of attracting regions. Let 
$O_n\subset {\bf R}$ be the basin of the attracting region $U_n$. 
As  the (entire) stable manifold of $x$ cuts every 
vertical segment, one deduces 
that $O_n$ is a open dense subset of ${\bf R}$ for every $n$.  
Then $\bigcap_{n\geq 0} O_n$ is the basin of the quasi-attractor 
which is a residual subset of ${\bf R}$. 

\end{proof}

%!TEX root = yetwild.tex

\section{Ejection of flexible periodic points}

In this section, we review some of the results of \cite{BS}. 
Then, by using these results, we finish the proof of 
Theorem~\ref{t.wild}.

\subsection{Flexible points: definition}
In \cite{BS} the authors defined the notion of 
\emph{flexible periodic points} and 
proved their generic existence under certain partially hyperbolic 
setting which includes partially hyperbolic filtrating Markov
partitions. Let us recall it.

A \emph{($\,2$-dimensional) periodic linear cocycle of period $n$} is a 
map
$\cA\colon (\ZZ/n\ZZ) \times \RR^2 \to (\ZZ/n\ZZ) \times \RR^2 $ 
of the form $\cA(i,p)=(i+1,A_i(p))$, where 
$A_i\in GL(2,\RR)$.
For linear cocycles $\mathcal{A}$, $\mathcal{B}$  of period $n$,  we define the \emph{distance} between 
$\mathcal{A}$ and $\mathcal{B}$, denoted by $\mathrm{dist}(\mathcal{A}, \mathcal{B})$, as follows:
\[
\mathrm{dist}(\mathcal{A}, \mathcal{B}) := 
\sup \| \mathcal{A}(x) - \mathcal{B}(x) \|,
\]
where $x$ ranges over all unit vectors in all fibers, that is, all unit 
vectors in $\coprod_{i \in\ZZ/n\ZZ}(\{ i\} \times \mathbb{R}^2)$.
This defines a distance on the space of linear cocycles of period $n$.

Let $\mathcal{A}_t$ denote a continuous one-parameter family of cocycles, 
that is, a continuous map from an interval to the space of cocycles. 
We define  the \emph{diameter of  the family $\mathcal{A}_t$} as
\[
\mathrm{diam}(\mathcal{A}_t) := 
\sup_{s< t}  \mathrm{dist} (\mathcal{A}_s, \mathcal{A}_t).
\]
To define the notion of flexible points, we first
give the definition of \emph{flexible cocycles}.
\begin{defi}\label{d.flexible}
Let $\cA=\{A_i\}_{i\in\ZZ /n\ZZ }$, $A_i\in \mathrm{GL}(2,\RR)$ 
be a linear cocycle of period $n > 0$ and 
let $\varepsilon$ be a positive real number.
We say that $\cA$ is \emph{$\varepsilon$-flexible} if there is a continuous 
one-parameter family of linear cocycles $\cA_t=\{A_{i,t}\}_{i\in\ZZ /n\ZZ}$ 
defined for $t \in [-1, 1]$ such that the following holds:
\begin{itemize}
\item $\mathrm{diam}(\mathcal{A}_t)  < \varepsilon$;
\item $A_{i,0}=A_i$, for every $i\in\ZZ /n\ZZ$;
\item the product matrix 
$A_{-1} := (A_{n-1,-1})\cdots (A_{0, -1})$ is a homothety
(remember that a square matrix is called \emph{homothety} if it has 
the form $\lambda\, \mathrm{Id}$, where $\lambda$ is a positive constant);
\item for every $t\in (-1,1)$, the product matrix 
$A_t = (A_{n-1,t})\cdots (A_{0,t})$ has 
two distinct positive contracting eigenvalues;
\item let $\lambda_t$ denote the smallest eigenvalue of 
 the product matrix $A_{t}$. Then 
 $\max_{-1 \leq t \leq 1} \lambda_t <1$;
\item $A_1$ has a real positive eigenvalue equal to $1$. 
\end{itemize}
\end{defi}

Now let us see the definition of flexible periodic points.
\begin{defi}
Let $M$ be a closed $3$-manifold,  $f \in \mathrm{Diff}^1(M)$ 
and ${\bf R} =\bigcup R_i$ be a partially 
hyperbolic filtrating Markov partition of $f$ (of saddle or attracting type).
A periodic orbit $\cO(x) \subset {\bf R}$ 
 is called \emph{$\varepsilon$-flexible} 
if the two dimensional linear cocycle of the restriction of the derivatives  
to the center-stable bundle along its orbit is  $\varepsilon$-flexible.
A periodic point is called $\varepsilon$-flexible if its orbit is so.
\end{defi}

\begin{rema}
To define the notion of flexibility of periodic points, we do not need the 
notion of partially hyperbolic filtrating Markov partitions
(partially hyperbolic structure is enough). 
Since in this paper we only treat flexible points 
appearing in partially hyperbolic filtrating Markov partitions, 
we adopted the definition above.
\end{rema}

\subsection{Isolation of flexible points} \label{s.perturbations}

In this subsection, we state a perturbation result on the isolation of 
flexible periodic points (see Proposition~\ref{p.kicking}). 

In \cite{BS}, the authors
showed that an $\varepsilon$-flexible
periodic point has great freedom of changing the position of 
their (strong) stable manifold by an $\varepsilon$-perturbation 
which is supported in an arbitrarily small neighborhood of the orbit.
By using such flexibility of flexible points (whose precise meaning 
will be elucidated in the next subsection), we can prove the following:

 \begin{prop}\label{p.kicking} Let ${\bf R} = \bigcup R_i$ be 
 a partially hyperbolic filtrating Markov partition 
 (of saddle or attracting type) of a diffeomorphism 
 $f$. Let $\varepsilon_0$ be a positive real number such that ${\bf R} = \bigcup R_i$ is still 
 a partially hyperbolic filtrating Markov partition for every
 $g$ which is $\varepsilon_0$-$C^1$-close
  to $f$ (see Lemma~\ref{l.continuation}).  
 
 Fix a positive real number $\varepsilon$ satisfying $0<\varepsilon<\varepsilon_0$.
 Suppose that  
 $x\in {\bf R}$ is a hyperbolic periodic point of $s$-index $2$ 
 with large stable manifold and $\varepsilon$-flexible. 
  Then, there is an $\varepsilon$-$C^1$-perturbation 
 $g$ of $f$, supported in an arbitrarily small 
 neighborhood of the orbit of $x$ such that the following holds:
 \begin{itemize}
  \item $f$ and $g$ coincide on the orbit of $x$. 
  In particular, $x$ is a periodic point for $g$ with the same orbit.
  \item $x$ is an $s$-index $1$ hyperbolic periodic point of $g$, 
  having an unstable eigenvalue arbitrarily close to $1$.
  \item the chain recurrence class of the orbit of $x$ is trivial for $g$. 
 \end{itemize}
 \end{prop}

We will see how we deduce Proposition~\ref{p.kicking}
from the results of \cite{BS} in the rest of this section.
As an immediate corollary of this proposition, we have the following:   
 \begin{coro}\label{c.kicking} Under the  hypotheses of Proposition~\ref{p.kicking}, 
 an  $\varepsilon$-perturbation $h$ of $g$ 
 turns
 $x$ into an $s$-index $2$ hyperbolic periodic point with trivial chain recurrence class. 
 \end{coro}
 \begin{proof} Proposition~\ref{p.kicking} produces a diffeomorphism $g$ for which 
 $x$ is an $s$-index $1$ hyperbolic periodic point of $g$, 
  having an unstable eigenvalue arbitrarily close to $1$ and a trivial chain recurrence class. This implies
  that  the orbit of
  $x$ admits a filtrating neighborhood $U$ in which the orbit of $x$  is the maximal invariant set. 
  Since $x$ has an unstable 
  eigenvalue very close to $1$, there is a $C^1$-small perturbation, supported in 
  an arbitrarily small neighborhood of $x$ (hence in $U$), such that the maximal invariant set in $U$ 
  is  the orbit of a periodic segment on which the 
  chain recurrent set consists of $3$ periodic orbits, $2$ of them 
  being $s$-index $1$ and the rest is the orbit of  $x$ (which has $s$-index $2$).
 \end{proof}

The proof of Proposition~\ref{p.kicking} will be done by
realizing a perturbation given by \cite[Theorem~1.1]{BS}  
and \cite[Corollary~1.1]{BS}
in a $2$-dimensional setting,
which is an extraction of the dynamics in the center-stable disc 
passing through $x$. 
Indeed, the proof of Proposition~\ref{p.kicking} is an
almost immediate corollary of Corollary~1.2 of \cite{BS}.
Because of the importance of this 
step and for the convenience of the reader, 
we will present the argument here again.
% we recall the precise result of \cite{BS} in the next subsection. 

\subsection{Background from \cite{BS}: the two dimensional setting}
\label{ss.twodim}

Now we cite Theorem~1.1
and Corollary~1.1 of \cite{BS},
which is used in the proof of Proposition~\ref{p.kicking}.

To state them, let us consider the following general setting.
Let  $S$ be a surface (a smooth two dimensional manifold, 
which may have non-empty boundary and may be non-connected)
and $F\colon S\to S$ be a $C^1$-diffeomorphisms on its image.  
Let $\varepsilon>0$ and 
 $x\in S$ be a periodic point of $F$ with period $\pi(x)$.
Suppose that $x$ is an $\varepsilon$-flexible 
(attracting) periodic point, that is, the differential $DF$
along the orbit of $x$ is an $\varepsilon$-flexible cocycle. 

Let $D$ be an attracting 
periodic disc of period ${\pi(x)}$, that is, 
$D$, $F(D),\dots, F^{{\pi(x)}-1}(D)$ are pairwise disjoint 
and $F^{\pi(x)}(D)$ is contained in the interior of $D$. 
We assume that $D$ is contained in the local stable manifold of $x$.
As in Lemma~\ref{l.TD}, 
the orbit space of $F^{\pi(x)}$ in $D\setminus \{ x\}$ is diffeomorphic to
the torus $\TT_D$. In the following, we denote this orbit 
space by $\TT_{F}$.
Then, $\partial D$ induces a homotopy 
class called parallel and the strong stable 
separatrices of $x$ induce two homotopy classes of curves
on the orbit space $\TT_{F}$
(recall that the flexibility assumption implies that 
$x$ has two distinct positive contracting eigenvalues).
We call this homotopy class {\it meridian}. 

We consider perturbations $G$ of $F$ preserving the 
 orbit of $x$, 
 supported in a small neighborhoods of $\mathcal{O}(x)$ 
such that $G^{\pi(x)}$ coincides with $F^{\pi(x)}$
 on $D\setminus F^{\pi(x)}(D)$. 
 By $\La_G$ 
 we denote the maximal invariant set of $G^{\pi(x)}$ in $D$
(i.e., $\La_G :=\bigcap_{n\geq 0} G^{n\pi(x)}(D)$) and
% As explained in Section~\ref{s.perturbations}, 
by $\TT_{G}$ we denote
the orbit space of $G^{\pi(x)}$ on $D \setminus \La_G$.
One can see that $\TT_{G}$ is diffeomorphic to the torus as well.
Then, for $G$ satisfying these conditions,
we identify $\TT_{F}$ and $\TT_{G}$
through the (unique) conjugacy 
between $F^{\pi(x)}|_{D\setminus \{x\}}$ 
and $G^{\pi(x)}|_{D\setminus \La_G}$
which is the identity map on 
$D\setminus F^{\pi(x)}(D)=D\setminus G^{\pi(x)}(D)$.

Under this identification, the freedom of flexible points
can be formulated by
the following Theorem~\ref{t.flexible},
which is proved in \cite[Theorem~1.1]{BS}:

\begin{theo}\label{t.flexible}
Let $F$ be a $C^1$-diffeomorphism of a surface $S$,
$x$ an $\varepsilon$-flexible periodic point
and $D$ an attracting periodic disc of period ${\pi(x)}$. 
Assume that $D$ contains $x$ and 
is contained in the stable manifold of $x$. 
Let $\gamma=\gamma_1 \cup \gamma_2 \subset \TT_{F}$ 
be the two simple closed curves which $W^{ss}(x)$ projects to.

Then, for any pair of $C^1$-curves $\sigma=\sigma_1\cup \sigma_2$ 
embedded in $\TT_{F}$ which 
is isotopic to $\gamma_1 \cup \gamma_2$,
there is an $\varepsilon$-perturbation $G$ of $F$, 
supported in an arbitrarily small neighborhood of the orbit of $p$
such that $G$ satisfies the following:
\begin{itemize}
\item $G$ coincides with $F$ along the orbit of $x$
(in particular $x$ is a periodic point of $G$ with the same period $\pi(x)$);
 \item $x$ is a (non-hyperbolic) periodic attracting point having an eigenvalue 
$\lambda_1\in]0,1[$ and an eigenvalue $\lambda_2=1$;
\item $D$ is contained in the basin of $x$;
\item the strong stable manifold $W^{ss}(x, G)$ projects to 
$(\sigma_1 \cup \sigma_2) \subset \TT_{G} \simeq \TT_{F}$.
\end{itemize}
\end{theo}

\begin{rema}
In \cite[Theorem~1.1]{BS}, the theorem is stated 
for $C^1$-diffeomorphism of surfaces. Since the perturbation 
we obtain is a local one, as we stated above,
the same result is also true for local diffeomorphisms 
(diffeomorphisms on their images).
\end{rema}

The orbit of $x$ for $G$ is a non-hyperbolic attracting periodic point, 
having an eigenvalue equal to $1$. 
By an extra, arbitrarily small perturbation, 
we can change the index of $x$ such that the strong stable 
manifold becomes the new stable manifold. 
Therefore we have the following (see Corollary~1.1 in \cite{BS}).
\begin{coro}\label{c.flexible} Under the hypotheses of Theorem~\ref{t.flexible},  
there is an $\varepsilon$-perturbation $H$ of $F$, supported 
in an arbitrarily small neighborhood of $x$
such that the following holds:
\begin{itemize}
 \item $x$ is a periodic saddle point having two real eigenvalues $0<\lambda_1<1<\lambda_2$, with
 $\lambda_2$ arbitrarily close to $1$;
\item $W^s(x)\setminus \cO(x)$ is disjoint from the maximal invariant set $\Lambda_H$;
\item $W^{s}(x, H)$ projects to 
$(\sigma_1 \cup \sigma_2) \subset \TT_{H}\simeq \TT_F$.
\end{itemize}
 \end{coro}
 
 \subsection{Ejecting a flexible periodic point:  proof of Proposition~\ref{p.kicking}}
 Now, by  Corollary~\ref{c.flexible}
 (which is proved by Theorem~\ref{t.flexible} and),
 let us finish the proof of Proposition~\ref{p.kicking}.

  \begin{proof}[Proof of Proposition~\ref{p.kicking}] 
Let ${\bf R} = \bigcup R_i$ be a partially hyperbolic Markov partition of 
a diffeomorphism $f$ and 
$x$ an $\varepsilon$-flexible periodic point of period $\pi(x)$
with large stable manifold.  
Let $D$ be the center-stable disc containing $x$. 
Note that by Lemma~\ref{l.large},
each point $f^i(x)$ also has large stable manifold for every $i$.

According to Lemma~\ref{l.obstacle}, 
the intersection  $D\cap f^{\pi(x)}({\bf R})$ is a union of
finite number of disjoint discs, 
one of them being $f^{\pi(x)}(D)$. 
The projection 
 $\Delta_D\subset$ of 
 $(D\cap f^{\pi(x)}({\bf R}))\setminus f^{\pi(x)}(D)$ to the torus  $\TT_{D}$ 
 is a union of pairwise disjoint $C^1$-discs. 
Thus every homotopy class of simple closed curves in 
$\TT_D$ contains curves disjoint from this union. 
 We fix two disjoint simple curves 
 $\sigma_1,\sigma_2\subset \TT_D$ disjoint from $\Delta_D$ and 
 isotopic to the meridian 
 (i.e., the projection of the strong stable separatrices of $x$). 

 We apply Corollary~\ref{c.flexible} to the restriction 
 $F$ of $f$ to the surface $S$ which is the union of center-stable discs 
 passing through points $f^i(x)$, $i\in\{0,\dots,\pi(x)-1\}$
 (notice these discs are mutually distinct, because each $f^i(x)$
 has large stable manifold).
 We obtain an $\varepsilon$-perturbation
 $H$ of $F$ satisfying the following conditions:
 \begin{itemize}
  \item $H$ coincides with $F$  on the orbit of $x$ and 
  out of an arbitrarily small neighborhood of the orbit of $x$ 
  (in particular $H^{{\pi(x)}}$ coincides with $F^{{\pi(x)}}$ 
  on $D\setminus   F^{\pi(x)}(D)$ and as a result 
  the orbit space $\TT_{H}$ is identified with $\TT_{F}$,
  as explained in the last subsection);
  \item $x$ is a saddle fixed point of $H^{\pi(x)}$ whose stable separatrices 
 are disjoint from the 
maximal invariant set of $H^{\pi(x)}$ in $D$ and 
their projections on $\TT_{F}= \TT_{H}$ are 
$\sigma_1$ and $\sigma_2$. 
\end{itemize}

 Then we realize $H$ as an $\varepsilon$-perturbation $h$ of $f$, 
 which coincides with $H$ 
 on the surface $S$  and with $f$
 out of an arbitrarily small neighborhood of 
 the orbit of $x$. 
 We take $H$ with its support sufficiently small such that 
 $h^{\pi(x)}({\bf R})=f^{\pi(x)}({\bf R})$ holds.
 By construction, 
 the stable separatrices of $x$ for $h$ are 
 disjoint from $(D\cap h^{\pi(x)}({\bf R}))\setminus h^{\pi(x)}(D)$ in the 
 fundamental domain 
 $D\setminus   H^{\pi(x)}(D)=D\setminus  F^{\pi(x)}(D) $.
 
 By the choice of $\varepsilon$,  
 the compact set 
 ${\bf R}$ is still a partially hyperbolic filtrating Markov partition for $h$ and 
 the point $x\in D$ is a hyperbolic saddle of $s$-index $1$ 
 of $h$ whose stable separatrices are disjoint from
 $(D\cap h^\pi({\bf R}))\setminus h^\pi(D)$. 
 Now Proposition~\ref{p.obstacle} implies 
 that the chain recurrence class of the orbit of $x$ is trivial, 
 which concludes the proof.  
 \end{proof}

 \subsection{Existence of flexible points with large stable manifolds }

In order to prove Theorem~\ref{t.wild} 
by Proposition~\ref{p.kicking} and Corollary~\ref{c.kicking}, 
the only thing we need to prove is
the existence of flexible points with large stable manifold
under the assumption of Theorem~\ref{t.wild}. 
We discuss it in this subsection. 

\begin{lemm}\label{l.flexiblelarge} 
Let ${\bf R} = \bigcup R_i$ be a partially hyperbolic filtrating Markov 
partition of a diffeormophism $f$ and
$\cU$ a $C^1$-neighborhood of $f$ in which one can find a 
continuation of ${\bf R}$ (see Lemma~\ref{l.continuation}).
Assume that there are  hyperbolic periodic points $p, p_1, q\in {\bf R}$ 
varying continuously with respect to $f\in \cU$ such that 
for every $g\in \cU$ they satisfy the following:
\begin{itemize}
 \item $p_g$ is of $s$-index $2$ and has large stable manifold;
 \item ${p_{1,g}}$ is of $s$-index $2$ and has complex (non-real) stable eigenvalues homoclinically related with $p$;
  \item $q_g$ is of $s$-index $1$ and $C(p_g)=C(q_g)$. 
\end{itemize}

Then, for any $\varepsilon>0$, there is a $C^1$-open and dense subset $\cD$ of $\cU$ such that every  
diffeomorphism $f\in \cD$ has a periodic point 
$x\in C(p)$ of $s$-index $2$, with large stable manifold, $\varepsilon$-flexible, 
homoclinically related with $p$, and whose orbit is $\varepsilon$-dense in $C(p)$.
\end{lemm}

\begin{proof} 
First notice that, if $x$ satisfies the announced properties for some $0<\varepsilon'<\varepsilon$ and a diffeomorphisms
$f\in\cU$ then there is a $C^1$-neighborhood of $f$ such that the continuation of $x$ satisfies the announced properties 
for the same $\varepsilon$. 
Thus, it is enough to prove the $C^1$-density of $f$ (in $\cD$) having a periodic point 
satisfying the properties as claimed. 

Remember that
\cite[Theorem 2]{BS} already announced the generic existence  of 
flexible periodic points whose orbits are $\varepsilon$-dense
in the chain recurrence class of $p$. 
The novelty here is that we require
these flexible points have large stable manifolds.  
Therefore we cannot apply directly the statement of  \cite[Theorem 2]{BS}, 
and we need to go back to two important steps of its proof.

First, \cite[Proposition 6]{BS} (essentially based on \cite{ABCDW}) asserts that any $f\in \cU$ is 
$C^1$-approximated by 
a diffeomorphism $g$ for which 
$p$ is homoclinically related with a periodic point
$y$ with the following properties:
\begin{itemize}
\item  $y$ has a stable eigenvalue arbitrarily close to $1$;
\item the smallest Lyapunov exponent of $y$ is strictly 
 smaller than a given negative constant $\lambda_f<0$ 
 (which only depends on $f$);
\item  the orbit of $y$ is arbitrarily close  (say $\varepsilon/3$-close)
to the chain recurrence class $C(p)$ in the Hausdorff distance.
\end{itemize}

Then, let us consider 
a hyperbolic basic set $\Lambda$ 
containing $y$ and (the continuations of) $p$ and $p_1$. 
As the orbit of $y$
is $\varepsilon/3$-dense in $C(p)$, it also holds for the hyperbolic set $\La$. By the upper semicontinuity of the chain recurrence class, 
this $\varepsilon/3$-density in $C(p)$ persists under small perturbations. 

This hyperbolic basic set $\Lambda$ has the stable bundle of dimension $2$.
Since $p_1$ has non-real stable eigenvalues,
this stable bundle does not admit any dominated splitting.
Now \cite[Proposition 7]{BS} asserts 
that $\Lambda$ admits arbitrarily small perturbations such that 
the continuation of $\Lambda$ has 
$\varepsilon$-flexible points $x_n$ whose 
orbits are arbitrarily close to $\Lambda$ in the Hausdorff distance 
and in particular are $\varepsilon$-dense in $C(p)$.  

Now we can complete the proof by Lemma~\ref{l.large2}:
the orbit of $x_n$ passes through a small neighborhood of 
$p$ which has large stable 
manifold.  Therefore Lemma~\ref{l.large2} ensures that the orbit of $x_n$ has a point with large stable manifold, and Lemma~\ref{l.large} ensures that this property holds for every point in the orbit of $x_n$. 
\end{proof}
%the continuous dependence of the local stable manifolds of the points in $\Lambda$  
%with respect to the point and with the diffeomorphisms: therefore any point passing close to 
%the point $p$ has a local stable 
%manifold which is close to the one of $p$.  As $p$ is assume to have a large stable manifold, 
%one gets that every
%periodic point close enough to $p$ has a large stable manifold too.  
%As the property of having a large stable manifold is 
%a property of the orbit, one gets that every periodic passing close enough to 
%$p$ has a large stable manifold, concluding the proof. 

\subsection{End of the proof of Theorem~\ref{t.wild}}

Now let us finish the proof of Theorem~\ref{t.wild}
by Proposition~\ref{p.kicking}, Corollary~\ref{c.kicking} and 
Lemma~\ref{l.flexiblelarge}.

\begin{proof}[Proof of Theorem~\ref{t.wild}]
We consider  a diffeomorphism $f$    
with a filtrating partially hyperbolic
Markov partition ${\bf R} = \bigcup R_i$ of either saddle type or attracting type such that there are:
\begin{itemize}
 \item  an $s$-index $2$ periodic point $p\in {\bf R}$ with large stable manifold,
\item  an $s$-index $2$ periodic point $p_1$ homoclinically related with $p$ 
having a complex (non-real) stable eigenvalue, and
\item  a periodic point $q$ of $s$-index $1$ which 
is $C^1$-robustly in the chain recurrence class $C(p,f)$.
\end{itemize}

% First, notice that, up to make an arbitrarily small perturbation, 
% we can assume that $p_1$ has a large stable manifold.
% This can be seen as follows: 
% First, since $p$ and $p_1$ are homoclinically realted,
% we consider a hyperbolic basic set containing $p$ and $p_1$.  
% Up to make an arbitrarily small perturbation (see \cite{BDP}), 
% we can find a periodic point in this basic set 
% with a complex (non real) stable eigenvalue 
% whose orbits passes arbitrarily close to
% $p$.  Such an orbit has f stable manifold according to 
% Lemma~\ref{l.large2}.

We apply Lemma~\ref{l.flexiblelarge} to $f$: 
for any $\varepsilon>0$, 
an arbitrarily small perturbation produces $\varepsilon$-flexible points
homoclinically related with $p_1$ (hence with $p$), 
having large stable manifold, 
and whose orbit is $\varepsilon$ dense in $C(p_1)=C(p)$. 

Then Proposition~\ref{p.kicking} and Corollary~\ref{c.kicking} 
allow us to create the announced points 
$x_g$ and $y_g$ (of $s$-index $1$ and $2$-respectively) with trivial chain recurrence classes 
and whose orbits are $2\varepsilon$
close to $C(p, f)$ with respect to the Hausdorff distance.
\end{proof}

% !TEX root =  yetwild.tex

\section{Examples}

In this section, we prove Proposition~\ref{p.exa},
that is, we show how we construct examples 
which satisfy the hypothesis of Theorem~\ref{t.wild}
with volume hyperbolicity.
We also briefly discuss the proof of Corollary~\ref{c.wild1}
at the end of this section.

\subsection{Outline of the construction}
The construction is done by the following two Propositions.

\begin{prop}\label{p.exphmp}
For every closed $3$-manifold $M$ 
there exists a diffeomorphism $f$ of $M$ 
having a mixing hyperbolic filtrating Markov partition 
${\bf R} = \bigcup R_i $ (of saddle type or attracting type) such that 
the following holds:
\begin{itemize}
\item ${\bf R}$ contains a non-trivial homoclinic class $H(p)$
where $p$ is a hyperbolic fixed point of $f$ in ${\bf R}$;
\item $p$ has large stable manifold.
% \item The linear endomorphism 
% $Df^{\mathrm{per}(p)}(p)|_{E^{s}}$ 
% has two distinct contracting real eigenvalues.
\end{itemize}
Such examples can be taken so that the maximal invariant set in ${\bf R}$ is 
a transitive hyperbolic set of $s$-index two. In particular, it admits a uniformly 
hyperbolic splitting $E^s \oplus E^u$ with $\dim(E^s)=2$.
\end{prop}
\begin{rema}\label{r.exphmp}  Indeed, for proving Proposition~\ref{p.exphmp}, 
we will construct an attracting ball $B^3$ containing the announced mixing hyperbolic filtrating Markov partition: 
then we embed it into any $3$-manifold.
Furthermore, for the attracting type, we can take the Markov partition in such a way that 
the whole ball is contained in the basin of attraction of it.
\end{rema}

To state the second proposition, 
recall that two hyperbolic periodic points $p_1, p_2$ of a 
diffeomorphism is said to have a \emph{heterodimentional cycle} if they have different 
indices and $W^s(p_1) \cap W^u(p_2)$ and $W^s(p_2) \cap W^u(p_1)$ are both non-empty.

\begin{prop}\label{p.pericre}
Let ${\bf R} = \bigcup R_i $ be a partially hyperbolic filtrating Markov partition of
saddle type or attracting type satisfying the conclusion of 
Proposition~\ref{p.exphmp}. Then, we can find a diffeomorphism $g$ 
which also satisfies two conditions in the conclusion of 
Proposition~\ref{p.exphmp}
(with the same ${\bf R} = \bigcup R_i $ being a partially hyperbolic filtrating Markov partition)  
and the following conditions:
\begin{enumerate}
\renewcommand{\labelenumi}{(\arabic{enumi})}
\item There exists a hyperbolic periodic point $p_1$ homoclinically related with $p$ such that
$Dg^{\pi(p_1)}(p_1)|_{E^{s}}$ has two contracting 
complex (non-real)
eigenvalues (where $\pi(p_1)$ denotes the period of $p_1$). 
\item There exists a hyperbolic periodic point $q$ of $s$-index $1$ 
which has a heterodimensional cycle associated with $p$.
\end{enumerate}
%Furthermore, the support of the perturbation $g\circ f^{-1}$ 
%(i.e., the compact set 
%$\overline{\{ x \in M \mid f(x) \neq g(x) \}}$) 
%can be taken so that it does not touch 
%$W_{\mathrm{loc}}^s(p)$, where $W_{\mathrm{loc}}^s(p)$
%is the connected component of $W^s(p) \cap C$ containing $p$. 
Furthermore, $g$ can be taken so that its maximal invariant set in ${\bf R}$
is volume hyperbolic.
\end{prop}

Let us first see how these two propositions conclude the 
construction of the examples. 
In the construction we use the 
following lemma from \cite{BDK} (see Theorem~1 of \cite{BDK}).
\begin{lemm}\label{l.BDK}
Let $f$ be a diffeomorphism having a heterodimensional cycle 
between two hyperbolic periodic points whose difference of stable indices is one. 
If one of them has non-trivial homoclinic class, 
then, by $C^1$-arbitrarily small perturbation, we can find a 
diffeomorphism such that (the continuations of) two periodic points 
belong to the same chain recurrence class $C^1$-robustly.
\end{lemm}
\begin{proof}[Proof of Proposition~\ref{p.exa}]
First, by Proposition~\ref{p.exphmp}, for any $3$-manifold
we take a diffeomorphism $f$ with 
a partially hyperbolic (indeed, uniformly hyperbolic) filtrating Markov partition ${\bf R} = \bigcup R_i $
either saddle or attracting type 
satisfying the conclusion of Proposition~\ref{p.exphmp}.
Then, by applying Proposition~\ref{p.pericre}  to $f$, 
we obtain $g$ having a hyperbolic periodic point $p_1$  homoclinically related with $p=p_g$ and having
complex eigenvalues,  and a hyperbolic periodic point $q$ 
that has a heterodimensional cycle with $p = p_{g}$.
Then, by applying Lemma~\ref{l.BDK}
to the heterodimentional cycle between $p_{g_2}$ and $q$, 
we  get a diffeomorphism which satisfies all the previous properties 
and furthermore $q$ and $p$ being 
robustly in the same chain recurrence class.
This gives the announced diffeomorphism. 
\end{proof}

\subsection{Example of filtrating (partially) hyperbolic Markov partition: proof of Proposition~\ref{p.exphmp}}

The proof of Proposition~\ref{p.exphmp} follows by well known arguments, so we only give the sketch of it.
In the following, as announced in Remark~\ref{r.exphmp}, we construct attracting endomorphisms (which are diffeomorphisms to their images)
on three dimensional ball $B^3$ to itself containing the partially hyperbolic filtrating Markov partition.
The construction on given manifold can be done by extending it as a diffeomorphism.
 
For the construction of the example of saddle type, we 
take a structurally stable diffeomorphisms on 
the sphere $S^2$ whose non-wandering set consists of
the Smale's horseshoe, one source and one sink. 
By multiplying it by a contraction on the interval, one gets an attracting region diffeomorphic to 
$S^2\times[-1,1]$, which can be seen as a neighborhood of the boundary of the ball $B^3$. 
One can complete the dynamics on $B^3$ by adding a source. 
Note that one can embed this dynamics on any $3$-manifold. 
This gives us a mixing (hyperbolic) filtrating Markov partition of saddle type. 

For the attracting type, we take the Smale's solenoid 
attractor defined on a solid torus. 
It is not difficult to take such a diffeomorphism on $B^3$ 
so that $B^3$ is contained in the basin of attraction of the solid torus
(see the construction in \cite{BLY} based on \cite{Gi}).

To complete the construction of attracting type, 
we need an additional argument on the 
shape of rectangles. Remember that we require that 
for any non-empty intersection of two rectangles 
the lid boundary of one of the rectangles is contained in 
the interior of that of the other. In order to get this property 
we  make a tricky choice of rectangles as follows:
suppose $F: S^1 \times \mathbb{D}^2 \to S^1 \times \mathbb{D}^2$,
$F(\theta, x) = (2\theta, F_\theta(x))$ is a Smale's solenoid map. 
Then we divide $S^1 \times \mathbb{D}^2$ into four pieces
$P_i := 
\{(\theta+(i/4), x) \mid \theta \in [0, 1/4], x \in \mathbb{D}^2\}$,
$i = 0, 1, 2, 3$. Then we shrink $P_1$ and $P_3$ a little bit 
in $\mathbb{D}^2$ direction. Then the family of rectangles $\{ P_i\}$ gives 
the desired example. 

\begin{rema}
In the argument above, we started from the Smale's solenoid attractor. 
Note that the same construction works for 
any Williams type attractor (see \cite{W}).
\end{rema}

\subsection{Complex eigenvalues and heterodimensional cycle: the proof of Proposition~\ref{p.pericre}}
Let us prove Proposition~\ref{p.pericre}.
The proofs of (1) and (2) have similar structures. We 
take a periodic point and modify 
the behavior of the diffeomorphism near the periodic point 
in the center-stable direction 
to obtain the desired property. The modification itself is simple but 
we need careful control the differentials in order to guarantee that the modification 
does not destroy the partially or volume hyperbolic structure.
In this subsection, we prove Proposition~\ref{p.pericre} except the 
fact that the modification can be done preserving the volume hyperbolicity.
Since the proof of it involves subtle linear algebraic argument, we 
will discuss it in the next subsection.

\subsubsection{Auxiliary results}
We prepare auxiliary lemmas which will be used in the proof.
The following lemma guarantees the existence of convenient coordinates
around a periodic point  in a partially hyperbolic filtrating Markov partition.

\begin{lemm}\label{l.coflat}
Let ${\bf R} = \bigcup R_i $ be a partially hyperbolic filtrating Markov partition 
(of saddle or attracting type) of $f$,
$q \in {\bf R}$ a hyperbolic periodic point of $s$-index 2 and $\mathcal{C}^u$ a 
strictly invariant unstable cone field over ${\bf R}$.
Then, there exists a $C^1$-coordinate neighborhood $(\varphi, U)$
with a coordinate system $(x, y, z)$ around $q$ such that the following holds:
\begin{itemize}
\item the local stable manifold of $q$ is equal to the $xy$-plane;
\item the local unstable manifold of $q$ is equal to the $z$-axis;
\item for every $\bar{x} \in U$, the cone field $(D\varphi)(\mathcal{C}^u_{\bar{x}})$ 
contains $z$-direction and is transverse to the $xy$-plane;
\item for every $\bar{x} \in U$, 
the cone field $(D\varphi)\left(Df(\mathcal{C}^u_{f^{-1}(\bar{x})})\right)$ 
is well-defined and contains the $z$-direction
(note that because of the invariance we also know that this cone field is 
transverse to the $xy$-plane).
\end{itemize}
\end{lemm}
The proof of this lemma is easy. So we omit it.

%
%\begin{proof}\marginpar{Trivial? Just say a word.}
%Given $q \in {\bf R}$, we take a rectangle $R_i$ which contains $q$.
%Then, we take coordinates $(X, Y, Z)$ so that 
%$R_i$ has image $\mathbb{D} \times I$ and the image of the cone field 
%containing $Z$-direction and does not contain $X, Y$-direction.
%By a coordinate change given by a parallel transformation, we can 
%assume that $q$ is the origin of the coordinate.
%In these coordinates, the local unstable manifold is a $C^1$-disc which may fail to be flat. 
%Thus, we take another coordinates system as follows:
%In $(X, Y, Z)$-coordinates, since $D$ is transversal to the cone field, 
%the center-stable disc is given as a graph $\{\left(X, Y, h(X, Y)\right) \mid (X, Y) \in \mathbb{D}^2 \}$,
%where $h$ is a $C^1$-function.
%We introduce  new coordinates $(\tilde{X}, \tilde{Y}, \tilde{Z}) = (X, Y, Z - h(X, Y))$. In these coordinates, 
%the local stable manifold coincides with the $\tilde{X}\tilde{Y}$-plane. 
%Also, since the differential of the change of coordinates  preserves the $\tilde{Z}$-direction, 
%the image of the cone field in this coordinate 
%still contains the $\tilde{Z}$-direction and 
%is tranversal to $\tilde{X}\tilde{Y}$-plane.
%
%By performing similar coordinate change, we can take a coordinate $(x, y, z)$ in which 
%not only the local stable one but also the local unstable manifold is flat. Finally, by restricting this coordinate
%to sufficiently small neighborhood of $q$, we can guarantee the conditions on the cone field.
%\end{proof}

\begin{rema}\label{r.strong}
In Lemma~\ref{l.coflat}, if $q$ has two different real contracting eigenvalues, we can choose the coordinate
so that the local strong-stable manifold of $q$ coincides with 
the $x$-axis. Furthermore, 
by taking a linear coordinate change and restricting it to a small neighborhood of $q$, 
we can also assume that the weak stable direction at $q$ is
equal to the $y$-direction.
\end{rema}

The following lemma guarantees the existence of 
a function having convenient behavior for the construction.

\begin{lemm}\label{l.chi}
For every $K>1$ and $\delta \in (0, 1)$, there exists 
$\alpha_0 \in (0, 1)$ such that for any 
$0<\alpha<\alpha_0$ there is 
a smooth diffeomorphism $\chi=\chi_{_{K,\delta,\alpha}} : [-1, 1] \to [-1 ,1]$ satisfying the following:
\begin{itemize}
\item $\chi (z) = K z$ for $z \in [-\alpha, \alpha]$;
\item $ 1 - \delta \leq \dot{\chi} (z) \leq K$ for every $z \in [-1, 1]$, where $\dot{\chi}(z)$ is 
the first derivative of $\chi(z)$;
\item for $z$ near $\pm 1$, $\chi (z) = z$;
\item $\max_{z \in [-1, 1]} |\chi (z) - z| \leq \alpha K$.
\end{itemize} 
\end{lemm}
\begin{proof}
Take a continuous piecewise linear function 
\[ P(z) =
\begin{cases}
Kz  & (z \in [0, \alpha]), \\
((1-K\alpha)/(1-\alpha))(z-1)+1 & (z \in [\alpha, 1]), 
\end{cases}
\]
and define $P(z) :=-P(-z)$ for $z \in [-1, 0]$.
If $\alpha$ is sufficiently small,
we can see that $P$ satisfies all the conditions we claimed except at $z = \pm\alpha$ and near $z=\pm1$.
Then, by removing the corner at $z = \pm\alpha$ and flatten it around $z=\pm1$ appropriately, we obtain the function.
\end{proof}

\subsubsection{Local modifications (I): preservation of cones}
Proposition~\ref{p.pericre} announces the 
existence of modifications in the center-stable direction keeping 
the partially hyperbolic and volume hyperbolic structures.  In this subsection we give a 
general modification result preserving these structures.
The modifications  
will be done in local coordinates: in this section we work on $\mathbb{R}^3$.

First, we  fix a bump function 
$$\rho(t) : \mathbb{R} \to [0, 1], \quad \mbox{ satisfying }\quad
\rho|_{[-1/2, 1/2]} = 1 \; \mbox{ and }\; \rho(t) = 0 \; \mbox{ for }\;|t| > 1.$$
We fix a constant $C_\rho$ which is an upper
bound of the derivative of $\rho$.
%\item We define $\Psi(x, y, z) :  \mathbb{R}^3 \to \mathbb{R}$ by 
%$\Psi(x, y, z) = \rho (r) \rho (z)$, where $r = \sqrt{x^2 +y^2\,}$.

We start from the following definition.

\begin{defi} We say that a one-parameter family $\{ \Gamma_t\}_{t\in[0,1]}$ of diffeomorphisms of $\RR^2$ 
is a \emph{modification family} if it satisfies the following properties: 
 
\begin{itemize}
 \item every $\Gamma_t$ is supported in  the unit disc $\mathbb{D}^2$;
 \item  the family $\{\Gamma_t\}_{t\in [0,1]}$ is jointly smooth with respect to $x, y$ and $t$;
 \item $\Gamma_0=\mathrm{Id}_{\mathbb{R}^2}$;
 \item for every $t\in[0,1]$, $\Gamma_t$ preserves the origin, that is, $\Gamma_t (0, 0) = (0, 0)$ for every $t$.
\end{itemize} 
 
 We say that the modification family $\{\Ga_t\}$ is \emph{area preserving} if for every $t\in[0,1]$ the diffeomorphism
 $\Ga_t$ is area preserving.
 
\end{defi}

Given a modification family
$\{\Gamma_t\}_{t \in [0, 1]}$,
by putting $\tilde{\Gamma}(x, y, z) := (\Gamma_{\rho (z)}(x, y), z)$, we  obtain 
a diffeomorphism of $\mathbb{R}^3$, supported on $\mathbb{D}^2 \times [-1, 1]$ 
and $\tilde{\Gamma}|_{\mathbb{D}^2 \times \{ 0\}} = \Gamma_1$. 
However, it may be that $\tilde{\Gamma}$ has large differentials and 
that may cause the destruction of hyperbolic structures.
The following lemma shows that we can realize the behavior of $\Gamma_1$
in the center-stable direction keeping hyperbolic structures,
by adding an extra modification.

\begin{lemm}\label{l.modif}
Let $\cC^u_1,\cC^u_2$ be two cones of $\RR^3$ 
strictly containing the $z$-direction,  
transverse to the $xy$-plane, and $\cC^u_1$ is strictly contained in $\cC^u_2$.
Let $\{\Gamma_t\}_{t\in[0,1]}$ be a modification family.
Then, given $\eta >0$ there exists a diffeomorphism $\hat{\Gamma} :\mathbb{R}^3 \to \mathbb{R}^3$ 
such that the following holds:
\begin{enumerate}
\renewcommand{\labelenumi}{(\arabic{enumi})}
\item $\hat{\Gamma}$ is supported on $\DD^2\times [-1,1]$;
\item for every $\bar{x} \in \mathbb{R}^3$, $D\hat{\Gamma}(\bar{x})(\mathcal{C}^u_1)$ 
is strictly contained in $\cC^u_2$;
\item for every $\bar{x} \in \mathbb{R}^3$ and 
every unit vector $u \in \mathcal{C}^u_2({\bar{x}})$,  
we have $\| D\hat{\Gamma}(\bar{x})(u)\| > 1 -\eta$ (note that the same holds 
for every unit vector in $\mathcal{C}^u_1$);
\item $\hat{\Gamma}$ preserves the $xy$-plane (that is, the plane $\{z=0\}$).  
Furthermore, the restriction
$\hat{\Gamma}|_{\RR^2 \times \{ 0\}}$ is conjugated to  $\Gamma_1$ by a homothety of $\RR^2$;
\item $\hat{\Gamma}$ preserves the $z$-axis 
(note that this, together with the third item, implies that the 
restriction of $\hat{\Gamma}$ to the $z$-axis has its derivative greater than $1-\eta$).
\end{enumerate}
Furthermore, if  the modification family $\{\Gamma_t\}$ is  area preserving,
then we can choose $\hat{\Gamma}$ such that the following holds:
$$|J_{xy}\hat{\Gamma}-1| < \eta,$$
where $J_{xy}\hat{\Gamma}$ is the determinant of 
the differential restricted to the $xy$-plane.
\end{lemm}

\begin{proof}
First, remember that
$\tilde{\Gamma}(x, y, z) := \left( \Gamma_{\rho(z)}(x,y), z \right)$
is a smooth diffeomorphism of $\RR^3$ supported on the cylinder 
$\{(x, y, z) \mid r\leq 1, z\in[-1,1]\}$,  
where $r := \sqrt{x^2+ y^2}\,$. 
The diffeomorphism $\tilde\Ga$ preserves the $z$ coordinate, 
hence its derivative at each point preserves the $xy$-plane. 

We consider now the diffeomorphism of $\RR^3$ obtained by conjugating it by a homothety:
\[
\tilde{\Gamma}_{\varepsilon} = H_{\varepsilon} \circ \tilde\Ga \circ (H_{\varepsilon})^{-1},
\]
where $H_{\varepsilon}(x, y, z) := (\varepsilon x, \varepsilon y, \varepsilon z)$ 
and $\varepsilon$ is a constant in $(0, 1)$.
By definition, $\tilde{\Gamma}_{\varepsilon}$ is a diffeomorphism of 
$\mathbb{R}^3$ supported on the cylinder 
$\left\{r\leq \varepsilon, |z|\leq \varepsilon\right\}$. 
It preserves the $z$ coordinate and its derivative at a point $(x,y,z)$ is that of
$\tilde{\Gamma}$ at $\frac1\varepsilon(x,y,z)$.  
In particular, the derivative preserves the $xy$-planes and 
its norm is uniformly bounded independently of the choice of $\varepsilon$. 

% \begin{rema} 
% If $\{\Gamma_t\}$ is area preserving, then
% the restriction of $\tilde{\Gamma}_{\varepsilon}$ 
% to any plane $z= \mbox{const.}$  is a area preserving 
% diffeomorphism to itself.  This remark will used to check that our modification 
% preserved the property of being volume hyperbolic.
% \end{rema}

Now, for any $K>1$, $\delta\in(0,1)$, and small $\alpha>0$,   we define a map 
\[
L_{_{K,\delta,\alpha}}(x, y, z) := (x, y, \rho(r) \chi_{_{K,\delta,\alpha}} (z) + (1- \rho(r)) z),
\]
where $\chi_{_{K,\delta,\alpha}}$ is a map which satisfies all the properties 
in Lemma~\ref{l.chi}.

For any real number $K$, we denote the linear map $(x,y,z)\mapsto (x,y,Kz)$ by $E_K$. 
Note that the following holds:
\begin{itemize}
 \item $L_{_{K,\delta,\alpha}}$ coincides with $E_K$ on the cylinder 
$\left\{r\leq\frac 12, z\in[-\alpha,\alpha]\right\}$;
\item the derivative  $DL_{_{K,\delta,\alpha}}(\bar x)$ preserves 
the $z$-direction for every $\bar x \in \mathbb{R}^3$;
\item the derivative $DL_{_{K,\delta,\alpha}}(\bar x)$ 
preserves the $xy$-plane for $\bar x$ in the (infinite) cylinder 
$\left\{r\leq \frac 12\right\}$.
\end{itemize}

Then, given a modification family $\{ \Gamma_t\}$, we define 
a diffeomorphism $\hat{\Gamma}_{K, \delta, \alpha}$ as follows:
\[
\hat{\Gamma}_{K, \delta, \alpha}(x, y, z) := 
L_{K,\delta,\alpha} \circ \tilde{\Gamma}_{\alpha/2}.
\]

The map $\hat{\Gamma}_{K, \delta, \alpha}$ is 
a smooth diffeomorphism of $\RR^3$ supported on the cylinder 
$\left\{ r\leq 1, z\in[-1,1] \right\}$. 
By construction, 
we can see that $\hat{\Gamma}_{K, \delta, \alpha}$ 
satisfies the conditions (1), (4) and (5) in
Lemma~\ref{l.modif}.
Now we choose the parameters $K, \delta$ and $\alpha$ so that $\hat{\Gamma}_{K, \delta, \alpha}$
satisfies the rest of the properties. 
To see this, we calculate the derivative of $\hat{\Gamma}_{K, \delta, \alpha}$.

First, we calculate $D\tilde{\Gamma}_{\alpha/2}= D(H_{\alpha/2} \circ \tilde \Ga \circ (H_{\alpha/2})^{-1}).$
Since the conjugation by a homothety does not have any effect on the calculation of the derivative, it can be written 
as follows:
 \begin{equation*}
\left(
\begin{array}{c|c}
 \vspace*{0mm}\raisebox{-4pt}{{\large\mbox{{{\large\mbox{{$A$}}}}}}($\frac 2\alpha\bar x$)} &  
 \raisebox{-4pt}{{\large\mbox{{$B$}}}$(\frac 2\alpha\bar x$)}\\[2ex]
  \hline
  0\quad\, 0 & 1
\end{array}
\right),
 \end{equation*}
where $A :=D_{xy}\Gamma_{\rho(z)}$, 
$B := \dot{\rho}(z)(D_{z}\Gamma_{\rho(z)})$. Note that these two matrices have bounded 
norms for fixed choice of the family $\{\Gamma_t\}$ and $\rho$. 
By this calculation we can see that this differential leaves the $xy$-plane
invariant. Furthermore, we can also see that if $\Gamma_t$ is area 
preserving, then this differential preserves the area of the $xy$-plane. 

We calculate the derivative $DL_{_{K,\delta,\alpha}}(\bar x)$. It can be written as follows:
\[
\left(
\begin{array}{c|c}
  \raisebox{-6pt}{{\Large\mbox{{$\mathrm{Id}$}}}} & 0\\ [-1ex]
   &  0\\\hline
 C(\bar x) & d(\bar x)
\end{array}
\right)
\]
where $\mathrm{Id}$ denotes the two dimensional identity matrix,
$$C(\bar x)=\left(\partial_x (\rho(r)) (\chi(z) -z ),\, \partial_y (\rho(r)) (\chi(z) -z )\right),$$
and $d(\bar x)= \rho(r) \dot{\chi}(z)+ (1-\rho(r)).$  
By definition of $\chi$ and $C_{\rho}$, 
one can see that the following inequalities hold:
\[
\|C(\bar x)\|\leq 2\alpha K C_\rho, \quad  1-\delta\leq d(\bar x)\leq K.
\]
   
Now, we calculate
$D\hat{\Gamma}_{_{K,\delta,\alpha}}$.
\begin{itemize}
 \item If $\bar {x} \in \{r\leq \frac\alpha 2, |z|\leq \frac\alpha2\}$, then
 $\bar x$ is in the support of $\tilde{\Gamma}_{\alpha/2}$
and $DL_{_{K,\delta,\alpha}} = E_K$. 
Hence, the derivative $D\hat{\Gamma}_{_{K,\delta,\alpha}}$ is 
 \begin{equation}\label{e.support}
\left(
\begin{array}{c|c}
 \vspace*{0mm}\raisebox{-4pt}{{\large\mbox{{$A$}}}($\frac 2\alpha\bar x$)} &  
 \raisebox{-4pt}{{\large\mbox{{$B$}}}$(\frac 2\alpha\bar x$)}\\[2ex]
  \hline
  0\quad 0 & K 
\end{array}
\right).
 \end{equation}
\item   If $\bar {x} \in \{r\leq \frac\alpha 2, |z|\geq \frac\alpha2\}$,
then $\bar{x}$ is outside the support of $\tilde{\Gamma}_{\alpha/2}$.
Furthermore, we know that $\rho \equiv 1$ on this region. 
Thus the derivative $D\hat{\Gamma}_{_{K,\delta,\alpha}}$ is
\begin{equation}\label{e.tube}
\left(
\begin{array}{c|c}
  \raisebox{-6pt}{{\Large\mbox{{$\mathrm{Id}$}}}} & 0\\ [-1ex]
   &  0\\\hline
  0 \quad 0& \dot{\chi}(z)
\end{array}
\right). 
\end{equation}
Recall that by definition we have 
$ 1-\delta\leq \dot{\chi}(z)\leq K$.

\item If $\bar{x} \in \{ r\geq \frac \alpha2\}$, 
 $\hat{\Gamma}_{_{K,\delta,\alpha}}=L_{_{K,\delta,\alpha}}$. 
Hence the derivative 
$D\hat{\Gamma}_{_{K,\delta,\alpha}}$ is
\begin{equation}\label{e.outside}
\left(
\begin{array}{c|c}
  \raisebox{-6pt}{{\Large\mbox{{$\mathrm{Id}$}}}} & 0\\ [-1ex]
   &  0\\\hline
 C(\bar x)& d(\bar x)
\end{array}
\right).
\end{equation}
Recall that we have
$\|C(\bar x)\|\leq 2\alpha K C_\rho$ and $1-\delta\leq d(\bar x)\leq K$.
\end{itemize}

By these calculations, we can conclude the following:
\begin{clai}\label{c.differ}
 Let $\cC^u_1,\cC^u_2$ be two cones of $\RR^3$  strictly containing the $z$ direction,  
 transverse to the $xy$-plane.
 Assume that  $\cC^u_1$ is strictly contained in $\cC^u_2$.
 
Then, for any $\eta>0$, there is $K>1$  and $\delta_0\in(0,1)$ such that  the following holds:
for any $0<\delta<\delta_0$, there is $0<\alpha_0\leq \alpha(K,\delta)$ such that 
for any $0<\alpha<\alpha_0$ 
the diffeomorphism $\hat{\Gamma}_{_{K,\delta,\alpha}}$ satisfies the 
conditions (2) and (3) in the conclusion of Lemma~\ref{l.modif}. Namely: 
\begin{itemize}
 \item For any $\bar x\in \RR^3$, the image $D\hat{\Gamma}_{_{K,\delta,\alpha}}(\bar x)(\cC^u_1)$ is strictly contained in $\cC^u_2$;
 \item For any $\bar x\in \RR^3$ and any unit vector $u\in\cC^u_2$  the norm of 
 $D\hat{\Gamma}_{_{K,\delta,\alpha}}(\bar x)(u)$ is larger than $1-\eta$:
 $$\|D\hat{\Gamma}_{_{K,\delta,\alpha}}(\bar x)(u)\|>1-\eta.$$
\end{itemize}
\end{clai}

\begin{proof}[Proof of Claim~\ref{c.differ}] 
Let us check that the announced properties hold on each domain. 

\begin{itemize}
\item If $\bar x$ belongs to the support of $\tilde{\Gamma}_{\alpha/2}$, that is, the region 
$\{r\leq\frac\alpha2, |z|\leq\frac\alpha2\}$,
then the derivative $D\hat{\Gamma}_{_{K,\delta,\alpha}}(\bar x)$ is given by (\ref{e.support}).
As $A(\frac 2\alpha \bar x)$ and $B(\frac 2\alpha \bar x)$ are uniformly bounded 
(independently of the choice of $\alpha$), 
we see that $D\hat{\Gamma}_{_{K,\delta,\alpha}}(\bar x)(\cC^u_1)$ 
is strictly contained in $\cC^u_2$ for $K$ large enough. Indeed, as $K$ goes to $+\infty$, 
the image $D\hat{\Gamma}_{_{K,\delta,\alpha}}(\bar x)(\cC^u_1)$ tends to the $z$ direction. 
Furthermore, for $K$ large, any vector in $\cC^u_2$ is expanded, in particular, each unit vector 
has its image longer than $1 -\delta$.

\item If $ \bar {x} \in \{r\leq\frac \alpha2, |z|\geq\frac\alpha2\}$, 
then the derivative $D\hat{\Gamma}_{_{K,\delta,\alpha}}(\bar x)$ is 
given by (\ref{e.tube}). It is a diagonal matrix
having the identity matrix in the $xy$-direction, 
and in the $z$ direction it is a multiplication at least by $1-\delta$.  Thus by choosing 
$\delta$ small enough, we have that $D\hat{\Gamma}_{_{K,\delta,\alpha}}(\bar x)(\cC^u_1)$ is contained 
in a cone arbitrarily close to $\cC^u_1$, hence contained in $\cC^u_2$. 
Furthermore, the vectors in $\cC^u_2$ are expanded at least by $1-\delta$. Hence,
by choosing $\delta$ to be smaller than $\eta$, one gets the second condition. 

\item If $\bar {x} \in\{r\geq \frac \alpha2\}$, then the derivative is given by (\ref{e.outside}). 
For fixed $K$ and $\delta$, 
when $\alpha$ goes to $0$ the non-diagonal term $C(\bar x)$ tends uniformly to $0$ (indeed,
it is bounded by $2\alpha K C_\rho$). Thus for $\alpha$ close to $0$, 
the derivative of $\hat\Ga$ almost preserves the $xy$-plane.
The diagonal terms  are the identity on the $xy$-plane and $d(\bar x)$ in the $z$ direction 
which satisfies $1-\delta\leq d(\bar x)\leq K$.  
Thus we can get the conclusion as in the second item.
\end{itemize}
As a result, given a constant $\eta$, by choosing $K$, $\delta$ and $\alpha$ appropriately 
in the order explained above, 
we have the desired properties.
\end{proof}

Finally, if 
$\{\Gamma_t\}$ is area preserving, then by choosing 
$\alpha$ sufficiently small, we can assume that the determinant of  $\hat{\Gamma}$ restricted 
to the $xy$-plane is arbitrarily close to $1$. 
To be more precise, the term $C(\bar x)$ is the unique trouble 
for the determinant not being exactly $1$, but we have seen that $C(\bar x)$ 
goes to $0$ when $\alpha$ tends to $0$.
Thus the proof of Lemma~\ref{l.modif} is completed.
\end{proof}

% Finally, for $\varepsilon>0$ we define 
% $\cL_{_{K,\delta,\alpha,\varepsilon}}=H_\varepsilon\cL_{_{K,\delta,\alpha}}H^{-1}_\varepsilon$, it is a 
% diffeomorphisms of $\RR^3$ supported on $[-\varepsilon,\varepsilon]^3$ and which satisfies the conclusion of 
% Lemma~\ref{l.modif} if and only if the same holds for  $\cL_{_{K,\delta,\alpha}}$.

% \begin{rema}\label{r.modif}
% Recall 
% $\cL_{_{K,\delta,\alpha}}=L_{_{K,\delta,\alpha}}\circ H_{\frac\alpha2}\circ L_1\circ H_{\frac\alpha2}^{-1}$.  
% The only properties of $L_1$ we used is that $L_1$ is supported on the cyclinder $r\leq 1, |z|\leq 1$, $L_1$ 
% reserves the $z$ coordinate, and (as it is a diffeomorphism  with compact support) 
% that its derivative is uniformly bounded.   
% \end{rema}

\subsubsection{Local modifications (II): existence of convenient families}
In this subsection, we discuss the existence of modification families 
with convenient properties for the construction.

\begin{lemm}\label{l.family}
There exists a modification family $\{\Gamma_t\}$ of $\mathbb{R}^2$
which is area preserving and 
satisfies one of the following conditions:
\begin{enumerate}
\renewcommand{\labelenumi}{(\arabic{enumi})}
\item $\|\Gamma_t(x, y)\| = \|(x, y)\|$ for every $t \in [0, 1]$ and every $(x, y) \in \mathbb{R}^2$,
where $\| \, \cdot \,\|$ denotes the distance from the origin. Furthermore, 
$D_{xy}\Gamma_1(x, y)$ is a rotation matrix of angle $\pi/2$ for every $(x, y)$. 
\item $\Gamma_1(x, y)$ preserves $x$-axis and 
$\|\Gamma_1(x, 0)\| \leq \|(x, 0)\|$
for every $x$. Furthermore, $D_{xy}\Gamma_1(0, 0) = 
\begin{pmatrix}
\beta & 0 \\
0 & \beta^{-1}
\end{pmatrix},$ where $\beta >1$ is a constant (which we can choose).  
\end{enumerate}
\end{lemm}
\begin{proof}
For the construction of modification family $\{\Gamma_t\}$ satisfying 
the condition (1) and area preserving property, we take a Hamiltonian function $\Xi(x, y) = (x^2 -y^2)\rho(r)$,
where $\rho$ is a smooth  bump function defined in the previous section.
By $\{\Omega_t\}$ we denote the time $t$ map of corresponding Hamiltonian vector field. 
Now $\Gamma_t := \Omega_{(\pi/2)t}$ gives us the desired family.

For (2), consider the Hamiltonian function $\Xi(x, y) = xy\cdot\rho(r)$ and put $\Gamma_t := \Omega_{\log(\beta) t}$. 
\end{proof}

%Let $\zeta\colon \RR^2\to \RR$ be the smooth map $-xy\rho(r)$.  Notice that:
%\begin{itemize} 
%\item   $\zeta=0$  out of the unit disc
%\item  $\zeta=0$ on the $x$ axis
%\item the derivative $D\zeta (x,0)$ does not vanish if $x\in(-1,0)\cup(0,1)$
%\item $\zeta$ coincides with $-xy$ in a neighbborhood of $(0,0)$. 
%\end{itemize}
%
%Let $\Om$ be the hamiltonian vector field of $\zeta$ and $\{\Om_t\}$ its flow. Note that $\Om_t$ is a smooth 
%diffeomorphism of $\RR^2$ supported on the unit disc and preserving the area.
%
%For any $\beta>0$ let $N_{1,\beta}\colon \RR^3\to \RR^3$ defined by 
%$$N_1(x,y,z)= \Om_{\rho(z)\beta}$$
%Then $N_1$  is diffeomorphism of $\RR^3$  supported 
%on the cylinder $r\leq 1, |z|\leq 1$,   preserving the $z$ coordinate, and preserving the area in each plane 
%$z=\mbox{const.}$.
%
%We denote  $N_{\varepsilon,\beta}=H_\varepsilon\circ N_{1,\beta}\circ H_\varepsilon^{-1}$, and 
%$\cN_{_{K,\delta,\alpha,\beta}}= L_{_{K,\delta,\alpha}}\circ N_{\alpha/2,\beta}$. According to 
%Remark~\ref{r.modif}, for any fixed $\beta$, Lemma~\ref{l.modif} holds for the family $\cN_{_{K,\delta,\alpha,\beta}}$
%instead of $\cL_{_{K,\delta,\alpha}}$.
%
%Notice that the  eigendirection  of $\cN_{_{K,\delta,\alpha,\beta}}$  at the fixed point $(0,0,0)$
%are the $x$, $y$ and $z$ axis, and the eigenvalues are $e^{-\beta}$, $e^\beta$ and $K$ respectively. 

\subsubsection{Realizing the diffeomorphisms of $\RR^3$ as the modification of $f$}
Let us start the proof of Proposition~\ref{p.pericre}. 

\begin{proof}[Proof of (1)] 
Let ${\bf R} = \bigcup R_i$ be a partially hyperbolic filtrating Markov partition of $f$
satisfying all the hypotheses
and let $\mathcal{C}^u$ be a strictly $Df$-invariant unstable cone field on $\bf R$.
Since $H(p)$ is non-trivial and $\bf R$ is a filtrating set, there 
exists a hyperbolic periodic point $p_1 \in {\bf R}$ homoclinically related with $p$.
We can assume that $Df^{\pi(p_1)}(p_1)|_{E^s}$ has 
positive determinant (see for example \cite[Lemma 4.16]{BDP}). 
If $Df^{\pi(p_1)}(p_1)|_{E^s}$ has a non-real eigenvalue, 
then we are done. So suppose that it does not.
By applying Lemma~\ref{l.coflat}, we choose local 
coordinates $\varphi$  in a neighborhood $U$ of $p_1$ in which the 
local invariant manifolds of $p_1$ are flat. By taking a linear coordinate 
change around $p_1$, we can assume that 
the two different stable eigendirections 
of $Df^{\pi(p_1)}(p_1)|={E^s}$ are orthogonal. 

Let $\cC^u_2$ and $\cC^u_1$ be the image of the unstable cone $\cC^u(p_1)$ and of $Df(\cC^u(f^{-1}(p_1)))$
under  $D\varphi$, respectively.
Take $\eta>0$ sufficiently close to $1$ such that $(1-\eta)^{-1}$ is strictly less than a lower bound $\mu$ for  
the rate of the expansion of vectors in the unstable cone $\cC^u(p_1)$.  
We apply Lemma~\ref{l.modif} to these cones, the family of diffeomorphisms $\{ \Gamma_t\}$
in Lemma~\ref{l.family} (1) and $\eta$ to obtain the diffeomorphism $\hat{\Gamma}$.
%and we choose $K,\delta,\alpha$ given by the Lemma. 

Notice that the same conclusion of Lemma~\ref{l.modif} holds for 
any pair of cones sufficiently close to $\cC^u_1$ and $\cC^u_2$. 
We choose $\varepsilon>0$ such that
for any $\bar x\in[-\varepsilon, \varepsilon]^3$  the cones $D\varphi(\cC^u(\varphi^{-1}(\bar x)))$ and 
$D\varphi(Df(\cC^u(\varphi^{-1}(\bar x))))$ are close enough to $\cC^u_2$ and $\cC^u_1$ respectively
so that the conclusions of Lemma~\ref{l.modif} still holds.

Now we define $g_0$ to be the diffeomorphism which coincides with   
$\varphi^{-1}\circ (H_\varepsilon \circ  \hat{\Gamma}\circ (H_\varepsilon)^{-1}) \circ\varphi\circ f$ on $f^{-1}(U)$ 
and with $f$ outside. 
By construction, 
the cone field $\cC^u$ is strictly invariant under $Dg_0$ and $Dg_0$ expands the vectors by a factor 
at least $(1-\eta)\mu>1$. So the cone field $\cC^u$ is not only strictly invariant but also unstable. 
% Thus we proved that $\bf R$ is a partially hyperbolic filtrating Markov partition for $g_0$. 
Notice that for sufficiently small $\varepsilon$ the orbit of $p_1$ meets the support of 
$\varphi^{-1}\circ (H_\varepsilon \circ  \hat{\Gamma}\circ (H_\varepsilon)^{-1}) \circ\varphi$ only at the
point $p_1$.  Then a simple calculation shows that $p_1$ has stable complex eigenvalues.

Let us show that if $\varepsilon$  sufficiently small,
then $p_1$ keeps the homoclinic relation with $p$.
Indeed, in our coordinates, 
the local stable manifold $W^s_{\mathrm{loc}}(p_1,f)$ of $p_1$ is exactly equal to 
the $xy$-plane, which is preserved by this modification. 
For $\varepsilon$  small, 
the support of the modification does not intersect the images $f^i(W^s_{\mathrm{loc}}(p_1,f))$ for $0<i<\pi(p_1)$.  
Thus $W^s_{\mathrm{loc}}(p_1,f)=W^s_{\mathrm{loc}}(p_1,g_0)$ for every $\varepsilon$.  
Similarly, we have $W^u_{\mathrm{loc}}(p_1,f)=W^u_{\mathrm{loc}}(p_1,g_0)$.
Let $s^+\in W^s_{\mathrm{loc}}(p_1,f)\cap W^u(p,f)$ and $s^-\in W^u_{\mathrm{loc}}(p_1,f)\cap W^s(p,f)$
be heteroclinic points.  
For $\varepsilon$ small enough the support of modification is disjoint from the negative orbit of 
$s^+$ and of the 
positive orbit of $s^-$. Thus $s^+$ and $s^-$ are still heteroclinic points between 
$p$ and ${p_1}$.

Finally, since the support of the perturbation is contained in an arbitrarily small
neighborhood of $p_1$, this neighborhood can be chosen 
so that it is disjoint from   
$W^s_{\mathrm{loc}}(p,f)$, and hence $W^s_{\mathrm{loc}}(p,f)=W^s_{\mathrm{loc}}(p,g_0)$.
Thus, we can see that $p$ still has large stable manifold. 
Thus the proof is completed.
\end{proof}

Let us see the second modification.

\begin{proof}[Proof of (2)] 
We start from $f$ satisfying the conclusion of item (1). Remember that $p_1$ is a 
periodic point with (non-real) complex stable eigenvalues and $p$ is a hyperbolic periodic point 
homoclinically related with $p_1$ having large stable manifold.

According to \cite[Proposition 2.5]{BDP}, 
an arbitrarily $C^1$-small perturbation of $f$ produces a hyperbolic periodic point
$q$ of $s$-index $2$, homoclinically related with $p$ and $p_1$, and whose derivative in the 
period restricted to
the center-stable space is a contracting homothety. 
We fix a point of transversal heteroclinic intersection 
$x\in  W^s(q) \cap W^u(p)$. By an arbitrarily small perturbation of $f$ in an arbitrarily small neighborhood of
$q$ we can change slightly the derivative of $q$ in the period 
in such a way that it has two real eigenvalues of different 
moduli and (the continuation of) $x$ still 
belongs to the strong stable manifold  $W^{ss}(q)$ of $q$. 

Thus we now assume that $f$ itself has these properties. 
More precisely, $f$ has a periodic point $q$ homoclinically
related with $p$, with two stable real 
eigenvalues of different moduli, and a point $x$ of heteroclinic 
intersection belonging to the strong stable manifold of 
$q$.  

Then by applying Lemma~\ref{l.coflat}, together with Remark~\ref{r.strong},
we take a local coordinate $\phi\colon V \to \RR^3$  around $q$ such that
the local stable manifold coincides with the $xy$-plane, 
the local strong stable manifold coincides with the $x$-axis, 
the $y$ direction is the weak stable direction 
%(so that the stable space is the $xy$-plane) 
and the local unstable manifold coincides with the $z$-axis. 

Now we perform the modification as in the proof of (1):
we use Lemma~\ref{l.modif} for the cones 
 $D\phi(\cC^u(q))$ and 
$D\phi\left( Df(\cC^u(f^{-1}(q)))\right)$, $\eta$ as in (1) and 
$\{ \Gamma_t \}$ as in Lemma~\ref{l.family}(2) letting $\beta$ greater than the 
weak stable eigenvalue of $q$ to obtain the diffeomorphism $\hat{\Gamma}$ of $\mathbb{R}^3$.
Using this $\hat{\Gamma}$, we modify $f$ as follows:
$\phi^{-1}\circ (H_\varepsilon \circ  \hat{\Gamma}\circ (H_\varepsilon)^{-1}) \circ\phi\circ f$ on $f^{-1}(V)$ 
and keep intact outside. 

As $\beta$ is bigger than the inverse of the weak stable eigenvalue of $q$, 
we see that $q$ 
is now an $s$-index $1$ hyperbolic saddle.  By the similar argument as is in the proof of (1), 
by choosing sufficiently small $\varepsilon$, we can check the preservation of partially hyperbolic 
filtrating Markov partition structure, the largeness of the stable manifold through $p$
and the preservation of heteroclinic relation between $p$ and $q$. Once again by 
decreasing $\varepsilon$ if necessary, we can avoid the 
interference of the modification to the (transverse)
homoclinic intersection between $p$ and $p_1$. 
Thus for sufficiently small $\varepsilon>0$ the modification above gives us the desired diffeomorphism $g$.
\end{proof}

%Furthermore as the $x$-axis is preserved by $N_{_{K,\delta,\alpha,\varepsilon}}$, and as 
%$N_{_{K,\delta,\alpha,\varepsilon}}$ maps every point of this $x$-axis on a point which is not farther from the origin, 
%one gets that the $x$-axis is still include in the local (strong) stable manifold of $q_2$ (it is indeed now
%its stable manifold). 
%
%Recall that, when $\alpha$ tends to $0$ the support of the perturbations tends to $0$.  This allows us to check that
%the negative orbit of $x_-$  and the positive orbit of $x_+$ so that these orbits remains heteroclinic 
%connection between $q_2$ and $p$. 

\begin{rema}\label{r.c-zero}
While the modification above needs to be $C^1$-large in general, by shrinking $\varepsilon$, 
we can assume that this modification is arbitrarily $C^0$-small. 
\end{rema}

\subsection{Preservation of volume hyperbolicity}

In this section we will prove the last part of Proposition~\ref{p.pericre}.
That is, we show that if the initial diffeomorphism $f$ is volume hyperbolic 
then we can perform 
the modification keeping the volume hyperbolicity. 
This will follow from the following two properties: 
 \begin{itemize}
  \item the modification we  perform can be done 
  so that the unstable bundle is almost preserved (this can be
  guaranteed by replacing the unstable cone field $\cC^u$ 
  by high forward image $Df^n(\cC^u)$).
  \item the modifications we perform almost preserve  the area of the horizontal subspaces.
 \end{itemize}

\subsubsection{Normal bundles and volume hyperbolicity}
 Let us explain why these two properties guarantee the 
 preservation of volume hyperbolicity. 
 Since we already proved the invariance of the cone field 
 and the uniform expansion of the vectors in the unstable cone in the previous subsection, 
 we certainly have the partial hyperbolicity, in particular,
 uniform expansion property of the unstable bundle.  
 It remains to prove the uniform contraction 
 of the area in the center-stable bundle. 
 This is not so easy because we do not know the position of the center-stable bundle after the modification:
 the only fact available is that it is transverse to the unstable cone.  
 we need to see how we can control the center-stable 
 determinant without controlling precisely the center-stable plane.
 
 To see this, we consider the following general situation.
 Let $f$ be a diffeomorphism of $M$ and $K \subset M$ 
 be a compact $f$-invariant set. 
 We assume that 
 $K$ is a maximal invariant set of some neighborhood $U$ of $K$
 and it is partially hyperbolic with a uniformly expanding, one-dimensional  unstable bundle 
 and a two-dimensional center-stable bundle $T_KM =E^{cs} \oplus E^{u}$.  
 Furthermore, we assume that 
 $K$ is also volume hyperbolic, 
 that is, the determinant of $Df$ restricted to the 
 center-stable bundle is uniformly contracting.  
Let $\cC^u$ be a strictly invariant unstable cone field defined on $U$, 
strictly containing the unstable direction and transverse to the 
center-stable direction at every point of $K$.
By definition we also know that the vectors in $\cC^u$ are uniformly expanding. 
Such a cone field always exists, by shrinking the neighborhood $U$ if necessary
(see Lemma~\ref{l.conenbd}). 

% We will define now an abstract linear cocycle over $K$ as follows.  

Let $\cP=\{P(x) \mid x\in U, P_x\subset T_xM\}$ 
be a continuous distribution of the same dimension as that of
center-stable bundle of $K$, and is transverse to $\cC^u(x)$. 
Note that $\cP$ defines a vector bundle over $U$. 

Now for $x\in K$ we define a linear map
$D(P,f)(x)\colon P(x)\to P(f(x))$ to be 
the one obtained as the composition of 
$Df(x)|_{P(x)}\colon P(x)\to T_{f(x)} M$ and 
$\mbox{Proj}_{\parallel_{E^u(f(x))}}P(f(x))$
%  that is,
% the projection to $P_{f(x)}$ parallel to $E^u(f(x))$
(for two complementary vector subspaces $ V$ and $W$ in a vector space, 
by $\mbox{Proj}_{\parallel_{V}}W$ we denote the projection from $V \oplus W$
to $W$ along to $V$). 
The collection of maps  $D(P,f)(x)$ defines a linear cocycle 
$\cD_{_{\cP,f}}$ on the bundle obtained
restricting $\cP$ over $K$.

\begin{rema} The linear cocycle $\cD_{_{\cP,f}}$ is conjugated to the restriction 
of $Df$ to the center-stable bundle
$E^{cs}$ by a continuous bundle map inducing the identity map on the base space $K$
(i.e., a continuous family of linear maps $E^{cs}(x)\to P(x)$, $x\in K$).  
Indeed, this map is given by the projection along $E^{u}(x)$.

By the existence of the conjugation, 
we see that $\cD_{_{\cP,f}}$ is uniformly volume contracting. 
Thus we can choose a metric on $M$ and a 
constant $\lambda$ satisfying $0<\lambda<1$ such that for 
any $x\in K$ the determinant of the linear map $D(P,f)(x)$ with respect to orthonormal basis has  
absolute value bounded  by $\lambda$ from above. 
\end{rema}

\subsubsection{Modification and normal bundle}
Now, let us consider the effect of modifications in this setting.
Let $g=h\circ f$ be a diffeomorphism such that  $g$ strictly leaves the cone field $\cC^u$ invariant 
on $U$, expands uniformly the vectors in $\cC^u$.  
We assume that  the maximal invariant set $K_g$ of $g$ in $U$
is contained in a small neighborhood of  $K$.
%  (for being rigorous we should consider a sequence 
%  $g_n=h_n\circ f$ so that $K_{g_n}$ is contained
% in arbitrarily small neighborhood of $K$ for $n\to \infty$). 
We also assume that the unstable bundle $E^u_g$ is very 
close to $E^u= E^u_f$. More precisely, we require that 
for every $x\in K_g$ there exists a point $y\in K$ 
close to $x$ so that $E^u_g(x)$ is close  $E^u(x)$. 

Let us consider the situation where $h$ almost preserves the bundle $\cP$
% (that is $Dh_n(\cP)$ tends to $\cP$ uniformly), 
and that the determinant of the restriction $Dh|_{\cP}$ is very close to $1$. 
In this setting, we have the following:
\begin{clai}\label{p.volume}
Under the above hypotheses, $g$ is volume hyperbolic on $K_g$.
% (more precisely, 
% we can find a partially hyperbolic splitting on $K_g$, which is volume hyperbolic). 
\end{clai}
% If $E,F$ are supplementar subspaces of a vector space we will denote $\mbox{Proj}_{\parallel_E}F$ 
% the projection on $F$ parallel to $E$.
\begin{proof} 
We take a linear cocycle 
$\cD_{_{\cP,g}}=\{D(P,g)(x) \mid  x\in K_g\}$,
where \[D(P,g)(x) := 
\left( \mathrm{Proj}_{\parallel_{Df(E^u_g(x))}}P(g(x)) \right) \circ Dg(x)|_{P(x)}.
\] 
Fix some constant  $\lambda_1$ satisfying 
$\lambda< \lambda_1<1$.   We show that the determinant of  $D(P, g)(x)$ has 
its absolute value bounded by 
$\lambda_1$ for every $x\in K_g$.

Given $x \in K_g$, by assumption, 
there exists a point $y\in K$ close to $x$ such that $E^u_g(x)$ is close to $E^u_f(y)$.
Thus $Df(E^u_g(x))$  is close to $Df E^u_f(y)=E^u_f(f(y))$. 
Furthermore, since $x$ is close to $y$, we see that $P(f(x))$ is close to $P(f(y))$.
Combining these, we see that the map
$\left( \mathrm{Proj}_{\parallel_{Df(E^u_g(x))}}P(f(x)) \right) \circ Df|_{P(x)}$ 
is very close to $D(P,f)$ and therefore its 
determinant is almost bounded by $\lambda$. 

Now recall that $Dh$ almost preserves $\cP$.
This implies that $Dh(P(f(x)))$ is very close to $P((h\circ f)(x))=P(g(x))$.  
As a consequence, the determinant the
restriction of 
$\left( \mathrm{Proj}_{\parallel_{E^u_g(g(x))}}P(g(x))\right)$ 
to $Dh(P(f(x)))$
is very close to $1$. 
Furthermore, the determinant of the restriction of  $Dh$ to $P(f(x))$ is assumed to be almost $1$. 

As a result, we deduce that the determinant of 
\[
\left(\mathrm{Proj}_{\parallel_{E^u_g(g(x))}}P(g(x))\right) \circ Dh 
\circ \left(\mathrm{Proj}_{\parallel_{Df(E^u_g(x))}}P(f(x))\right) \circ Df|_{P(x)}
\]
is bounded by some constant $\lambda_1 <1$ 
(which can be chosen independently of $x$). 

Then, notice that we have the following equality:
\begin{clai}\label{c.proje} For every $x \in K_g$, we have
\[
D(P,g)(x)= 
        \left(\mathrm{Proj}_{\parallel_{E^u_g(g(x))}}P(g(x)) \right)  \circ Dh 
\circ \left(\mathrm{Proj}_{\parallel_{Df(E^u_g(x))}}P(f(x)) \right) \circ Df|_{P(x)}.
\]
\end{clai}
 \begin{proof}
 This is a consequence of the following the general fact: let $U_i=V_i\oplus W_i$, $i=1,2,3$ be vector spaces and 
 $F\colon U_1\to U_2$, $H\colon U_2\to U_3$ and $G\colon U_1\to U_3$ 
 are linear maps satisfying $G=H\circ F$.  Assume that 
 $F(W_1)=W_2$ and $H(W_2)=W_3$.  Then 
\[
\left( \mathrm{Proj}_{\parallel_{W_3}}V_3 \right) \circ G =
\left( \mathrm{Proj}_{\parallel_{W_3}}V_3 \right) \circ H \circ 
\left( \mathrm{Proj}_{\parallel_{W_2}}V_2 \right) \circ F. 
\]
The equality that we want to show is the direct consequence of this general result.
\end{proof}
Since the determinant of the right hand side of Claim~\ref{c.proje} is bounded by $\lambda_1$,
so is the left hand side.
\end{proof}

\subsubsection{Conclusion of volume hyperbolicity}
Now, let us finish the 
proof of Proposition~\ref{p.pericre}.
\begin{proof}[Proof of the last part of Proposition~\ref{p.pericre}]
We perform the modification $g= h \circ f$ of
Proposition~\ref{p.pericre} to higher forward refinement of ${\bf R} = \bigcup R_i$.
Since, by taking forward iteration, the unstable cone $\cC^u$ converges to the unstable 
direction at each point of the maximal invariant set of ${\bf R}$, 
we can assume that the modification in Proposition~\ref{p.pericre}
preserves the unstable direction as much as we want. 
Note that, by the last item of
Lemma~\ref{l.modif}, we can assume that $h$ almost preserves the area of 
center-stable planes. 

Now we take a two dimensional distribution $\cP$ transverse to $\cC^u$ which coincides with 
the $xy$-plane in the support of the modification. Remember that the modification
$g = h \circ f$ can be done so that it is arbitrarily $C^0$-close to $f$, see Remark~\ref{r.c-zero}. 
As a result, 
we can assume that the  maximal invariant set $K_g$  is contained in arbitrarily small 
neighborhood of that of $f$. Thus, Claim~\ref{p.volume} ensures that for $g = h \circ f$ 
chosen as above has  maximal invariant set which is volume hyperbolic.
\end{proof}

\subsection{On the proof of Corollary~\ref{c.wild1}}

Let us see how we prove Corollary~\ref{c.wild1}. 

In Proposition~\ref{p.exphmp}, Remark~\ref{r.exphmp} and Proposition~\ref{p.pericre}, 
we have seen that we can construct a local map $f$ on the three-dimensional 
ball $B^3$ satisfying the following:
\begin{itemize}
\item it is a diffeomorphism on its image;
\item the ball is attracting, that is, $f(B^3) \subset \mathrm{int}(B^3)$ holds;
\item $f$ contains a mixing 
partially hyperbolic Markov partition ${\bf R}$ of attracting type which
satisfies the hypothesis of Theorem~\ref{t.wild} having volume hyperbolicity
on the maximal invariant set of ${\bf R}$;
\item the whole ball $B^3$ is contained in the basin of attraction of ${\bf R}$, 
that is, for every $x \in B^3$ there exists $n>0$ such that $f^n(x) \in {\bf R}$ holds.
\end{itemize}
Note that, because of these properties, 
the chain recurrence class $C(p)$ (remember that $p$ is the hyperbolic 
periodic point with large stable manifold)
is the unique quasi-attractor in $B^3$ in $C^1$-robust way (see Corollary~\ref{c.quasi}).

Now, for any $3$-manifold $M$ we can construct a diffeomorphism 
without (topological)
attractors or repellers as follows (see also section~4.2 of \cite{BLY}): 
take a Morse function of $M$ and its gradient flow. Then replace each 
sink by $(B^3, f)$ and each source by $(B^3, f^{-1})$ and glue them appropriately. 
Now by Corollary~\ref{c.isolated} we have the conclusion: 
there are finitely many quasi-attractors 
and quasi-repellers and, 
for $C^1$-generic diffeomorphisms in the neighborhood, 
every quasi-attractor and quasi-repeller is wild.

% !TEX root = yetwild.tex

\vspace{1.5cm}

\begin{itemize}
\item[]  \emph{Christian Bonatti \quad} (bonatti@u-bourgogne.fr)
\begin{itemize}
\item[] Institut de Math. de Bourgogn\'{e} CNRS - URM 5584
\item[] Universit\'e de Bourgogne Dijon 21004, France
\end{itemize}
\item[] \emph{Katsutoshi Shinohara \quad} (herrsinon@07.alumni.u-tokyo.ac.jp)
\begin{itemize}
\item[] Department of Mathematics and Information Sciences,
\item[] Tokyo Metropolitan University, Hachioji, Tokyo, Japan
\end{itemize}
\end{itemize}

\end{document}